\documentclass[12pt]{article}

\usepackage{amsfonts}
\usepackage{amsmath}
\usepackage{amssymb}
\usepackage{amscd}
\usepackage{amsthm}
\usepackage{indentfirst}
\usepackage{MnSymbol}
\usepackage{stackrel}
\usepackage[vmargin=3.4cm,hmargin=2.3cm]{geometry}

\newtheorem{theorem}{Theorem}[section]
\newtheorem{corollary}[theorem]{Corollary}
\newtheorem{lemma}[theorem]{Lemma}
\newtheorem{proposition}[theorem]{Proposition}

\newtheorem*{theorem no number}{Theorem}

\theoremstyle{definition}

\newtheorem{definition}[theorem]{Definition}
\newtheorem{example}[theorem]{Example}
\newtheorem*{acknow}{Acknowledgements}
\newtheorem*{notation}{Notation}

\newtheorem{remark}[theorem]{Remark}

\newtheorem*{remark no number}{Remark}

\theoremstyle{remark}

\newtheorem*{claim no number}{Claim}

\newcommand{\R}{{\mathbb{R}}} 

\newcommand{\Q}{{\mathbb{Q}}}

\newcommand{\Z}{{\mathbb{Z}}}
\newcommand{\N}{{\mathbb{N}}}

\usepackage{zref-xr}

\usepackage{tikz-cd}
\usetikzlibrary{matrix,arrows,decorations.pathmorphing}

\usepackage{comment}

\newcommand{\bigslant}[2]{{\raisebox{.2em}{$#1$}\left/\raisebox{-.2em}{$#2$}\right.}}
\newcommand\restr[2]{{% we make the whole thing an ordinary symbol
  \left.\kern-\nulldelimiterspace % automatically resize the bar with \right
  #1 % the function
  \vphantom{\big|} % pretend it's a little taller at normal size
  \right|_{#2} % this is the delimiter
  }}
  
\usepackage{hyperref}
\usepackage{url}
\usepackage{graphicx}

\title{\textbf{Polytope Novikov Homology}}
\author{Alessio Pellegrini \\ \\ \textit{Department of Mathematics, ETH Z\"urich, Switzerland} \\ alessio.pellegrini@math.ethz.ch}

\begin{document}

\maketitle \date

\abstract{Let $M$ be a closed manifold and $\mathcal{A} \subseteq H^1_{\mathrm{dR}}(M)$ a polytope. For each $a \in \mathcal{A}$ we define a Novikov chain complex with a multiple finiteness condition encoded by the polytope $\mathcal{A}$. The resulting polytope Novikov homology generalizes the ordinary Novikov homology. We prove that any two cohomology classes in a prescribed polytope give rise to chain homotopy equivalent polytope Novikov complexes over a Novikov ring associated to said polytope. As applications we present a novel approach to the (twisted) Novikov Morse Homology Theorem and prove a new polytope Novikov Principle. The latter generalizes the ordinary Novikov Principle and a recent result of Pajitnov in the abelian case.}

\section{Introduction}

Given a closed manifold $M$ and a cohomology class $a \in H^1_{\mathrm{dR}}(M)$, one can define the so called \emph{Novikov homology} $\mathrm{HN}_\bullet(a)$, introduced by Novikov \cite{Novikov1981, Novikov1982}. Roughly speaking, $\mathrm{HN}_\bullet(a)$ is defined by picking a Morse representative $\alpha \in a$ and a cover on which $\alpha$ pulls back to an exact form $d\tilde{f}$, and then mimicking the definition of Morse homology using $\tilde{f}$ as the underlying Morse function. The groups $\mathrm{HN}_\bullet (a)$ enjoy three distinctive features: 
\begin{itemize}
\item \textbf{(Novikov-module)} The Novikov homology $\mathrm{HN}_\bullet (a)$ is a finitely-generated module over the so called \emph{Novikov ring} $\mathrm{Nov}(a)$. 
\item \textbf{(Cohomology-invariance)} The Novikov homology $\mathrm{HN}_\bullet(a)$ does not depend on the choice of Morse representative $\alpha$ of the prescribed cohomology class $a$.

\item \textbf{(Ray-invariance)} Morse forms on the same positive half-ray induce identical Novikov homologies: $\mathrm{HN}_\bullet(r \cdot a) \cong \mathrm{HN}_\bullet(\alpha)$ for all $r>0$.
\end{itemize}

The (twisted) Novikov Morse Homology Theorem says that $\mathrm{HN}_\bullet(a)$ is isomorphic to the twisted singular homology $H_\bullet\left(M,\underline{\mathrm{Nov}}(a)\right)$.\footnote{See Corollary \ref{corollary NHT for real} for a precise statement.} By using the Novikov Morse Homology Theorem one can thus investigate the relation between $\mathrm{HN}_\bullet(a)$ and $\mathrm{HN}_\bullet(b)$, when $a \neq b$, by studying $H_\bullet\left(M,\underline{\mathrm{Nov}}(a)\right)$ and $H_\bullet\left(M,\underline{\mathrm{Nov}}(b)\right)$ and their respective twisted coefficient systems $\underline{\mathrm{Nov}}(a)$ and $\underline{\mathrm{Nov}}(b)$ instead. The latter is a ``purely" algebraic task. 

In this article we refine the construction of Novikov homology $\mathrm{HN}_\bullet(a)$ and define what we call \emph{polytope Novikov homology} $\mathrm{HN}_\bullet(a,\mathcal{A})$ by including multiple finiteness conditions imposed by a polytope $\mathcal{A}=\langle a_0,\dots, a_k \rangle \subseteq H^1_{\mathrm{dR}}(M)$ containing $a$. These polytope Novikov homology groups $\mathrm{HN}_\bullet(a,\mathcal{A})$ retain the three features of $\mathrm{HN}(a)$ mentioned above, modulo replacing $\mathrm{Nov}(a)$ by a ``smaller" Novikov ring $\mathrm{Nov}(\mathcal{A})$. The Main Theorem in Section \ref{section Novikov Homology and Polytopes} gives a \emph{dynamical} relation between $\mathrm{HN}_\bullet(a,\mathcal{A})$ and $\mathrm{HN}_\bullet(b,\mathcal{A})$, i.e by staying in the realm of Novikov homology and not resorting to the algebraic counterpart of twisted singular homology.

\begin{theorem no number}[Main Theorem\footnote{See Theorem \ref{theorem main technical theorem} for the precise statement. Also note that the actual theorem contains a stronger chain level statement.}]
For every subpolytope $\mathcal{B} \subseteq \mathcal{A}$ and two cohomology classes $a,b \in \mathcal{A}$ there exists a commutative diagram
\begin{center}
\begin{tikzcd}
\mathrm{HN}_\bullet(a,\mathcal{A}) \arrow[d, hook] \arrow[r, "\cong"] & \mathrm{HN}_\bullet(b,\mathcal{A}) \arrow[d, hook] \\
\mathrm{HN}_\bullet \left( a,\mathcal{A}  |_{\mathcal{B}} \right) \arrow[r,"\cong"] & \mathrm{HN}_\bullet \left( b,\mathcal{A}  |_{\mathcal{B}} \right).
\end{tikzcd}
\end{center}
induced by continuation on the chain level.
\end{theorem no number}

The statement of the Main Theorem might be known to some experts in the field, but lacks a proof in the literature. Similar variants of the Main Theorem have been proved in different settings, most noteworthy are \cite{Ono2005, hutchings2008, groman2018symplectic, zhang2019}. For example, in \cite{Ono2005} Ono considers Floer-Novikov homology on a closed symplectic manifold\footnote{For the sake of simplicity we omit the precise conditions.} and proves:

\begin{theorem no number}[Ono \cite{Ono2005}]\label{theorem ono}
 If two symplectic isoptopies have fluxes that are close to each other, then their respective Novikov-Floer homologies are isomorphic. 
 \end{theorem no number}
The Novikov-Floer homologies mentioned in Ono's Theorem are defined over a common Novikov ring that takes into account several finiteness conditions simultaneously -- this modification is analogous to our implementation of polytopes. Within this analogy, the upper isomorphism in the Main Theorem corresponds to the isomorphism in Ono's Theorem, but with less assumptions: the nearby assumption of the fluxes in Ono's result would translate to a smallness assumption on $\mathcal{A}$, which is not needed. Let us mention that the formulation and setup of the Main Theorem comes closest to a recent result due to Groman and Merry \cite[Theorem 5.1]{groman2018symplectic}.

At the end of the paper we present two applications of the Main Theorem. In the first application we recover the aforementioned Novikov Morse Homology Theorem:\footnote{This is not a circular argument, since the Novikov Morse Homology Theorem is not used in Section \ref{section Novikov Homology and Polytopes}.} The proof, modulo details, goes as follows: taking $\mathcal{A}=\langle 0,a \rangle$, setting $\mathcal{B}=\langle a \rangle$, invoking the lower isomorphism in the Main Theorem, and unwinding the definitions reveals $$\mathrm{HN}_\bullet(a) \cong \mathrm{HM}_\bullet(f,\underline{
\mathrm{Nov}}(a)) ,$$ where the right hand side is Morse homology with local coefficients $\underline{\mathrm{Nov}}(a)$. The latter is known to be isomorphic to singular homology with twisted coefficients, for a quick proof see \cite[Theorem 4.1]{banyaga2019}, and thus we recover the Novikov Morse Homology Theorem. This line of reasoning is analogous to the proof of \cite[Theorem 5.3]{groman2018symplectic} and seems to be a novel approach to the Novikov Morse Homology Theorem: the proof draws a direct connection between Novikov and \emph{twisted} Morse homology instead of using the Novikov Principle and/or equivariant Morse homology, see \cite{latour1994, Pozniak1999, farber2004} for proofs of the Novikov Morse Homology Theorem using the latter.

The second application is concerned with a general \emph{polytope Novikov Principle}:\footnote{This is a slightly imprecise formulation, see Theorem \ref{theorem Polytop Novikov Principle} for the precise statement}

\begin{theorem no number}[Polytope Novikov Principle]
Let $\mathcal{B}\subseteq \mathcal{A}$ be a subpolytope. Then for every $a \in \mathcal{A}$ there exists a Morse representative $\alpha \in a$ such that
\begin{equation*}
\mathrm{CN}_\bullet\left(\alpha,\mathcal{A} |_{\mathcal{B}} \right) \simeq C_\bullet\left(\widetilde{M}_\mathcal{A}\right)\otimes_{\Z[\Gamma_\mathcal{A}]}\mathrm{Nov}(\mathcal{A}|_\mathcal{B}),
\end{equation*}
as Novikov-modules.
\end{theorem no number}

The proof idea is similar to the sketch above -- one relates the polytope Novikov complex to a a twisted Morse complex by including the $0$-vertex in the polytope $\mathcal{A}$ and using the Main Theorem. We call this the \emph{0-vertex trick} (cf. Lemma \ref{lemma 0 vertex trick}). To get from the twisted Morse complex to the equivariant singular chain complex we use a Morse-Eilenberg type result (cf. Lemma \ref{lemma Morse-Eilenberg}) and a chain homotopy equivalence $\mathrm{CM}(\tilde{h}) \simeq C_\bullet \left( \widetilde{M}_{\mathcal{A}} \right)$ over the group ring of deck transformations $\Z[\Gamma_{\mathcal{A}}]$. 

Immediate consequences of the polytope Novikov Principle include the ordinary Novikov Principle (cf. Corollary \ref{corollary ordinary NP}) and a recent ``conical" Novikov Principle \cite[Theorem 5.1]{Pajitnov2019}\footnote{See Theorem \ref{theorem Pajitnov} for a statement below.} in the abelian case (cf. Corollary \ref{corollary pajitnov NP}). 

\begin{remark no number}
Symplectic homology is a version of Floer homology well-suited to certain non-compact symplectic manifolds. In \cite{pellegrini2020} we combine ideas of Ono's Theorem, the magnetic case \cite{groman2018symplectic}, and of the present paper to construct a polytope Novikov symplectic homology, which is related to Ritter's twisted symplectic homology \cite{ritter2009}. The analogue of the Main Theorem remains true. Applications include Novikov number-type bounds on the number of fixed points of symplectomorphisms with prescribed flux on the boundary, and the study of symplectic isotopies of such maps.
\end{remark no number}

\begin{acknow}
I would like to thank Will Merry for all the helpful discussions and for encouraging me to flesh out the polytope picture in the case of Novikov homology. This work has been supported by the Swiss National Science Foundation (grant \# 182564).
\end{acknow}

\section{Novikov Homology and Polytopes}\label{section Novikov Homology and Polytopes}

\subsection{Definition and properties of ordinary Novikov homology}\label{subsection the chain complex....}
In this subsection we quickly recall the (ordinary) definition of the Novikov chain complex and its homology, together with some well known properties. The main purpose is to fix the notation for the remainder of the section. For a thorough treatment of Novikov homology, we recommend \cite{Sikorav1987, Pozniak1999, farber2004} and the recently published \cite{Banyaga2004}. For more details on the construction of the Novikov ring see for instance \cite[Chapter 4]{HoferSalamon1995}.

Fix once and for all a closed smooth oriented and connected finite-dimensional manifold $M$. For any Morse-Smale pair $(\alpha,g)$ one can define the \emph{Novikov chain complex} 
\begin{equation*}
(\mathrm{CN}_\bullet(\alpha,g), \partial_\bullet),
\end{equation*}
whose homology is called \emph{Novikov homology} of $(\alpha,g)$
\begin{equation*}
\mathrm{HN}_i(\alpha,g):=\bigslant{\ker \partial_i}{\mathrm{im} \, \partial_{i+1}}, \quad i \in \N_0.
\end{equation*}
It is a standard fact that two Morse-Smale pairs with cohomologous Morse forms induce isomorphic Novikov homologies, thus we shall write $\mathrm{HN}_\bullet (a)$ with $a=[\alpha]$ to denote the Novikov homology of pairs $(\alpha,g)$. 

\begin{notation}
Sometimes we will also omit the $g$ in the notation of the chain complex. Moreover, Latin lowercase letters, e.g. $a,b$, will typically denote cohomology classes, while the respective lowercase Greek letters are representatives in the corresponding cohomology classes, e.g. $\alpha \in a, \, \beta \in b$.
\end{notation}

Let us quickly recall the relevant definitions. Each cohomology class $a$ determines a period homomorphism $\Phi_a \colon \pi_1(M) \to \R$ defined by integrating any representative $\alpha \in a$ over loops $\gamma$ in $M$.\footnote{Cohomologous one-forms induce the same period homomorphism by Stokes' Theorem.} Denote by $\ker(a)$ the kernel of the period homomorphism $\Phi_a$ and let $\pi \colon \widetilde{M}_a \to M$ be the associated abelian cover, i.e. a regular covering with $\Gamma_a:=\mathrm{Deck}(\widetilde{M}_a) \cong \bigslant{\pi_1(M)}{\ker (a)}$. Then $\alpha$ pulls back to an exact form on $\widetilde{M}_a$, i.e. $\pi^*\alpha=d \tilde{f}_\alpha$ for some $\tilde{f}_\alpha \in C^{\infty}(\widetilde{M}_a)$. Define 
\begin{equation*}
 V_i(\alpha):= \bigoplus_{\tilde{x} \in \mathrm{Crit}_i(\tilde{f}_\alpha)} \Z \langle \tilde{x} \rangle, \quad i \in \N_0,
\end{equation*}
where $\mathrm{Crit}_i(\tilde{f}_\alpha)$ denotes the critical points of $\tilde{f}_\alpha$ with Morse index $i$. The $i$-th Novikov chain group $\mathrm{CN}_i(\alpha)$ can then be defined as the \emph{downward completion} of $V_i(\alpha)$ with respect to $\tilde{f}_\alpha$, which shall be denoted by 
\begin{equation}\label{equation downward completion for ordinary chain group}
\widehat{V}_i(\alpha)_{\tilde{f}_\alpha} \text{ or more concisely } \widehat{V}_i(\alpha)_\alpha.
\end{equation}
Explicitly, elements $\xi \in \mathrm{CN}_i(\alpha)$ are infinite sums with a \emph{finiteness condition} determined by $\tilde{f}_\alpha$:
\begin{equation*}
\xi=\sum_{\tilde{x} \in \mathrm{Crit}_i(\tilde{f}_{\alpha})} \xi_{\tilde{x}} \, \tilde{x} \in \mathrm{CN}_i(\alpha) \; \iff \; \forall c \in \R \colon \;  \lbrace \tilde{x} \, \big | \, \xi_{\tilde{x}} \neq 0 \in \Z, \; \tilde{f}_{\alpha}(\tilde{x})>c \rbrace \text{ is finite}.
\end{equation*}
The boundary operator is defined by counting Morse trajectories of $\tilde{f}_{\alpha}$:
\begin{equation*}
\partial \colon \mathrm{CN}_i(\alpha) \to \mathrm{CN}_{i-1}(\alpha), \quad \partial \xi:=\sum_{\tilde{x}, \, \tilde{y}} \, \xi_{\tilde{x}} \cdot \#_{\mathrm{alg}} \, \underline{\mathcal{M}}(\tilde{x},\tilde{y};\tilde{f}_{\alpha}) \, \tilde{y},
\end{equation*}
where 
\begin{equation*}
\mathcal{M}(\tilde{x},\tilde{y};\tilde{f}_{\alpha})=\lbrace \tilde{\gamma} \in C^{\infty}(\R,\widetilde{M}) \, \big | \, \dot{\tilde{\gamma}}+\nabla^{\tilde{g}} \tilde{f}_{\alpha}(\tilde{\gamma})=0, \, \tilde{\gamma}(-\infty)=\tilde{x}, \, \tilde{\gamma}(+\infty)=\tilde{y} \rbrace
\end{equation*} is the usual moduli space and $\underline{\mathcal{M}}(\tilde{x},\tilde{y};\tilde{f}_{\alpha})=\bigslant{\mathcal{M}(\tilde{x},\tilde{y};\tilde{f}_{\alpha})}{\R}$. Similarly, we denote by $\mathcal{M}(x,y;\alpha)$ and $\underline{\mathcal{M}}(x,y;\alpha)$ the moduli spaces downstairs. The $\#_{\mathrm{alg}}$ indicates the algebraic count, i.e. counting the Novikov trajectories with signs determined by a choice of orientation of the underlying unstable manifolds.

The \emph{Novikov ring} $\Lambda_\alpha$ associated to $\alpha \in a$ is defined as the \emph{upward completion} of the group ring $\Z[\Gamma_a]$ with respect to the period homomorphism $\Phi_a$, therefore
\begin{equation*}
\lambda=\sum_{A \in \Gamma_a} \lambda_A \, A \in \Lambda_\alpha \; \iff \; \forall c \in \R \colon \;  \lbrace A \, \big | \, \lambda_A \neq 0 \in \Z, \; \Phi_a(A) < c \rbrace \text{ is finite}.
\end{equation*}
The Novikov ring $\Lambda_\alpha$ does not depend on the choice of representative $\alpha \in a$, thus we shall write $\Lambda_a$. Moreover, $\Lambda_a$ acts on $\mathrm{CN}_\bullet(\alpha)$ in the obvious way. By fixing a preferred lift $\tilde{x}_j$ in each fiber of the finitely many zeros $x_j \in Z(\alpha):=\left\lbrace x \in M \, \big | \, \alpha(x)=0 \right\rbrace$, one can view $\mathrm{CN}_\bullet(\alpha)$ as a finitely-generated $\Lambda_a$-module:
\begin{equation}\label{equation ordinary chain complex is novikov ring module}
\mathrm{CN}_i(\alpha) \cong \bigoplus_{\tilde{x}_j \in \mathrm{Crit}_i(\tilde{f}_{\alpha})} \Lambda_a \langle \tilde{x}_j \rangle \, \text{ as Novikov ring modules}. 
\end{equation}
Another standard fact asserts that the boundary operator $\partial$ is $\Lambda_a$-linear and consequently the Novikov homology $\mathrm{HN}_\bullet (a)$ carries a $\Lambda_a$-module structure. The latter is implicitly using the fact that isomorphism of Novikov homologies for cohomologous Morse forms, which is suppressed in the notation $\mathrm{HN}_\bullet (a)$, is also $\Lambda_a$-linear.

\begin{remark}\label{remark Z_2 Novikov theory}
If $M$ is \emph{not} orientable one can still define a Novikov homology by replacing $\Z$ with $\Z_2$ in all the definitions above.
\end{remark}

\subsection{Novikov homology with polytopes}\label{subsection Novikov homology with polytopes}
We are now ready to refine the Novikov chain complex using polytopes -- this notion will be key for the proof all incoming theorems.

\begin{definition}\label{def polytope}
Given $a_0,\dots,a_k \in H^1_{\mathrm{dR}}(M)$, denote by
\begin{equation*}
\mathcal{A}=\langle a_0, \dots, a_k \rangle \subset H^1_{\mathrm{dR}}(M)
\end{equation*}
the \textbf{polytope} spanned by the \textbf{vertices} $\{a_l\}_{l=0,\dots,k}$, i.e. the set of all convex combinations 
\begin{equation*}
a=\sum_{l=0}^k c_l \cdot a_l \text{ with } c_l \in [0,1] \text{ and } \sum_{l=0}^k c_l=1.
\end{equation*}
\end{definition}
To any polytope $\mathcal{A}$ we associate a regular cover 
\begin{equation*}
 \pi \colon \widetilde{M}_{\mathcal{A}} \to M \text{ determined by } \mathrm{Deck}(\widetilde{M}_\mathcal{A}) \cong \bigslant{\pi_1(M)}{\displaystyle\bigcap_{l=0}^k \ker (a_l)},
\end{equation*}
and we shall abbreviate $$\Gamma_\mathcal{A}:=\mathrm{Deck}(\widetilde{M}_{\mathcal{A}}).$$
\begin{example}\label{example polytope consisting of point}
For the polytope $\mathcal{A}=\langle a \rangle$ the covering $\widetilde{M}_\mathcal{A}$ agrees with the abelian cover $\widetilde{M}_a$ associated to $a \in H^1_{\mathrm{dR}}(M)$. The same is true for any polytope $\mathcal{A}$, whose other vertices $a_l$ satisfy $\ker (a) \subseteq \ker (a_l)$.
\end{example}

The defining condition of $\widetilde{M}_{\mathcal{A}}$ ensures that each vertex $a_l$ pulls back to the trivial cohomology class, and so does every $a \in \mathcal{A}$, see Lemma \ref{lemma vertex determien CN polytope}. We write  $\tilde{f}_{\alpha} \in C^{\infty}(\widetilde{M}_{\mathcal{A}})$ to denote some primitive of $\pi^*\alpha$ for $\alpha$ a representative of $a \in \mathcal{A}$. Now we fix a smooth section 
\begin{equation*}
\theta \colon \mathcal{A} \longrightarrow \Omega^1(M), \quad a \mapsto \theta_a
\end{equation*}
of the projection of closed one-forms to their cohomology class. In other words, $\theta_a$ is a representative of $a$. This enables us to talk about a ``preferred" representative of each cohomology class in the polytope. 

For every polytope $a \in \mathcal{A}$ we define
\begin{equation*}
V_i(\theta_a,\mathcal{A}):= \bigoplus_{\tilde{x} \in \mathrm{Crit}_i\left(\tilde{f}_{\theta_a}\right)} \Z \langle \tilde{x} \rangle.
\end{equation*}
The subtle but crucial difference to $V_i(\theta_a)$ is that $\widetilde{M}_{\mathcal{A}}$ does not necessarily coincide with the abelian cover $\widetilde{M}_a$.

\begin{definition}\label{def novikov chain group with polytope}
Let $\mathcal{A}$ be a polytope with section $\theta \colon \mathcal{A} \to \Omega^1(M)$. Then the \textbf{(polytope) Novikov chain complex groups}  
\begin{equation*}
\mathrm{CN}_i(\theta_a,\mathcal{A}), \quad i \in \N_0,
\end{equation*}
are defined as the intersections of the downward completions of $V_i(\theta_a,\mathcal{A})$ with respect to any $\tilde{f}_{\beta} \colon \widetilde{M}_{\mathcal{A}} \to \R$ for $b \in \mathcal{A}$. In other words, with the notation of \eqref{equation downward completion for ordinary chain group}:
\begin{equation*}
\mathrm{CN}_i(\theta_a,\mathcal{A}):=\bigcap_{b \in \mathcal{A}} \widehat{V}_i(\theta_a,\mathcal{A})_{\beta}.
\end{equation*}
\end{definition}

\begin{remark}\label{remark indep of primitive and cohomol for finiteness condition}
Let $\beta \in b$ be any representative. The choice of primitive $\tilde{f}_{\beta}$ of $\pi^*\beta$ is unique up to adding constants and hence does not affect the finiteness condition. Additionally, two primitives $\tilde{f}_{\beta}$ and $\tilde{f}_{\beta'}$ induce the same finiteness condition for $\beta, \, \beta' \in b$. Indeed, $\tilde{f}_{\beta'}-\tilde{f}_{\beta}=h \circ \pi$ for some smooth $h \colon M \to \R$ with $dh=\beta'-\beta$. Since $M$ is compact we get $$\tilde{f}_{\beta}(\tilde{x})>c \implies \tilde{f}_{\beta'}(\tilde{x}) > \min_{z} \vert h(z) \vert +c \text{ and } \tilde{f}_{\beta'}(\tilde{x})>d \implies \tilde{f}_{\beta}(\tilde{x}) > d-\max_{z} \vert h(z) \vert,$$ hence the two finiteness conditions are equivalent. This justifies Definition \ref{def novikov chain group with polytope}.
\end{remark}

Unpacking Definition \ref{def novikov chain group with polytope} we see
\begin{equation}\label{equation multi finiteness condition}
\xi=\sum_{\tilde{x} \in \mathrm{Crit}_i\left(\tilde{f}_{\theta_a}\right)} \xi_{\tilde{x}} \, \tilde{x} \in \mathrm{CN}_i(\theta_a, \mathcal{A})  \iff \forall b \in \mathcal{A}, \forall c \in \R \colon \;  \#\lbrace \tilde{x} \, \big | \, \xi_{\tilde{x}} \neq 0, \; \tilde{f}_\beta(\tilde{x})>c \rbrace < +\infty,
\end{equation}
where it does not matter which primitives $\tilde{f}_\beta$ we use, cf. Remark \ref{remark indep of primitive and cohomol for finiteness condition}. The right hand side describes a finiteness condition that has to hold for all $b \in \mathcal{A}$, hence we will refer to it as the \emph{multi finiteness condition}.

\begin{notation}
In view of Remark \ref{remark indep of primitive and cohomol for finiteness condition} we shall write $$\mathrm{CN}_i(\theta_a,\mathcal{A})=\bigcap_{b \in \mathcal{A}} \widehat{V}_i(\theta_a,\mathcal{A})_b$$ from now on.
\end{notation}

In a similar fashion we can define yet another completion of $V_i(\theta_a,\mathcal{A})$ by taking the completion with respect to less one-forms.

\begin{definition}\label{definition novikov chain complex restricted}
Let $\mathcal{B} \subseteq \mathcal{A}$ be a subpolytope. Then we define 
\begin{equation*}
\mathrm{CN}_i\left(\theta_a,\mathcal{A}  |_{\mathcal{B}}\right):=\bigcap_{b \in \mathcal{B}} \widehat{V}_i(\theta_a,\mathcal{A})_b
\end{equation*}
the \textbf{restricted} (polytope) Novikov chain complex groups of $\mathcal{B}\subseteq \mathcal{A}$.
\end{definition}

By definition we get the inclusion
$$\mathrm{CN}_\bullet(\theta_a,\mathcal{A}) \subseteq \mathrm{CN}_\bullet\left(\theta_a, \mathcal{A}  |_{\mathcal{B}} \right), \text{ for all subpolytopes } \mathcal{B} \subseteq \mathcal{A}.$$
The next lemma asserts that $\mathrm{CN}_\bullet(\theta_a,\mathcal{A})$ is uniquely determined by the vertices of $\mathcal{A}$. In other words, one only needs to check the multi finiteness condition for the finitely many vertices $a_l$. This is a straightforward adaptation of \cite[Lemma 7.3]{zhang2019}.

\begin{lemma}\label{lemma vertex determien CN polytope}
Let $\theta \colon \mathcal{A} \to \Omega^1(M)$ be as above. Then
\begin{equation*}
\mathrm{CN}_i(\theta_a, \mathcal{A})=\bigcap_{b \in \mathcal{A}} \widehat{V}_i(\theta_a,\mathcal{A})_b=\bigcap_{l=0}^k \widehat{V}_i(\theta_a,\mathcal{A})_{a_l}=\bigcap_{l=0}^k\mathrm{CN}_i \left(\theta_a, \mathcal{A} |_{a_l}\right), \quad \forall a \in \mathcal{A}.
\end{equation*}
More generally, for every subpolytope $\mathcal{B} \subseteq \mathcal{A}$ spanned by $b_j=a_{l_j}$:
\begin{equation*}
\mathrm{CN}_i(\theta_a, \mathcal{A} |_{\mathcal{B}})=\bigcap_{j}\mathrm{CN}_i \left(\theta_a, \mathcal{A} |_{b_j}\right), \quad \forall a \in \mathcal{A}.
\end{equation*}
\end{lemma}

One can play a similar game with the Novikov rings:

\begin{definition}
 Define the \textbf{(polytope) Novikov ring}  
\begin{equation*}
 \Lambda_\mathcal{A}=\bigcap_{b \in \mathcal{A}} \widehat{\Z}[\Gamma_{\mathcal{A}}]^b,
\end{equation*}
where $\widehat{\Z}[\Gamma_{\mathcal{A}}]^b$ denotes the upward completion of the group ring $\Z[\Gamma_{\mathcal{A}}]$ with respect to the period homomorphism $\Phi_b \colon \Gamma_{\mathcal{A}} \to \R$. 
 Analogously, for every subpolytope $\mathcal{B} \subseteq \mathcal{A}$ we define the \textbf{restricted} polytope Novikov ring 
\begin{equation*} 
 \Lambda_{\mathcal{A} |_{\mathcal{B}}}=\bigcap_{b \in \mathcal{B}} \widehat{\Z}[\Gamma_{\mathcal{A}}]^b.
\end{equation*}
\end{definition} 
 
 As before we get
\begin{equation*}
\Lambda_{\mathcal{A}} \subseteq \Lambda_{\mathcal{A} |_{\mathcal{B}}}, \text{ for all subpolytopes } \mathcal{B} \subseteq \mathcal{A}.
\end{equation*}
The obvious analogue to Lemma \ref{lemma vertex determien CN polytope} holds for Novikov rings as well. These rings enable us to view the polytope Novikov chain complexes as finite Novikov-modules just as in the ordinary setting \eqref{equation ordinary chain complex is novikov ring module}.

Next we try to equip the groups $\mathrm{CN}_\bullet(\theta_a,\mathcal{A})$ with a boundary operators that turns them into a genuine chain complex. The obvious candidate would be 
\begin{equation}\label{equation boundary operator}
\partial_{\theta_a} \colon \mathrm{CN}_\bullet(\theta_a,\mathcal{A}) \to \mathrm{CN}_{\bullet-1}(\theta_a,\mathcal{A}), \quad \partial_{\theta_a} \xi := \sum_{\tilde{x}, \tilde{y}} \xi_{\tilde{x}} \cdot \#_{\mathrm{alg}} \, \underline{\mathcal{M}}\left(\tilde{x},\tilde{y};\tilde{f}_{\theta_a}\right) \, \tilde{y}.
\end{equation}
Note that the moduli space above actually also depends on a choice of metric $g$, and so does the boundary operator $\partial_{\theta_a}$. When we want to keep track of the metric we will write $\mathrm{CN}_\bullet(\theta_a,g,\mathcal{A})$. For restrictions $\mathcal{A}  |_{\mathcal{B}}$ we define the boundary operator analogously.

Formally, the definition of $\partial_{\theta_a}$ looks identical to the definition of $\partial$ on $\mathrm{CN}_\bullet(\alpha)$, and morally it is. However, there are two major differences. Firstly, the cover $\widetilde{M}_\mathcal{A}$ might differ from the abelian cover $\widetilde{M}_a$ of $a$. Secondly, it is not clear whether $\partial=\partial_{\theta_a}$ preserves the multi finiteness condition, i.e. whether $\partial \xi$ lies in $\mathrm{CN}_\bullet(\theta_a,\mathcal{A})$. Luckily, we will achieve this by replacing the original section $\theta$ with a perturbed section $\vartheta \colon \mathcal{A} \to \Omega^1(M)$ (cf. Theorem \ref{theorem perturbed section}). Whenever the chain complex is defined we make the following definition.

\begin{definition}
Let $\vartheta \colon \mathcal{A} \to \Omega^1(M)$ be a section such that $\left(\mathrm{CN}_\bullet(\vartheta_a,g_{\vartheta_a},\mathcal{A}),\partial \right)$ defines a chain complex. Then we call the induced homology \textbf{(polytope) Novikov homology} and denote it by
\begin{equation*}
\mathrm{HN}_\bullet(\vartheta_a,g_{\vartheta_a},\mathcal{A}) \text{ or more abusively  }\mathrm{HN}_\bullet(\vartheta_a,\mathcal{A}).
\end{equation*}
Analogously, we define
\begin{equation*}
\mathrm{HN}_\bullet\left(\vartheta_a,g_{\vartheta_a},\mathcal{A}  |_{\mathcal{B}} \right)=\mathrm{HN}_\bullet\left(\vartheta_a,\mathcal{A}  |_{\mathcal{B}}\right).
\end{equation*}
\end{definition}

\begin{remark}\label{remark Novikov-modules }
Analogously to ordinary Novikov homology, one can show that the Novikov homologies $\mathrm{HN}_\bullet(\vartheta_a,\mathcal{A})$ and $\mathrm{HN}_\bullet \left(\vartheta_a,\mathcal{A}  |_{\mathcal{B}} \right)$ are both finitely-generated modules over the Novikov rings $\Lambda_\mathcal{A}$ and $\Lambda_{\mathcal{A} |_{\mathcal{B}}}$, respectively, thus generalizing the Novikov-module property. This follows from the fact that the boundary operator \eqref{equation boundary operator} is $\Lambda_{\mathcal{A}}$-linear (and similarly for the restricted case).
\end{remark}

\subsection{Technical results for Subsection \ref{subsection Section Perturbations}  }

In this subsection we state and prove all the technical auxiliary results needed for the proof of Theorem \ref{theorem perturbed section}, which roughly speaking asserts the well-definedness of the polytope chain complexes and their respective homologies after modifying the section $\theta \colon \mathcal{A} \to \Omega^1(M)$ to a new section $\vartheta \colon \mathcal{A} \to \Omega^1(M)$.

\begin{notation}
For any (closed) one-form $\rho$ we will denote by $\nabla^g \rho$ the dual vector field to $\rho$ with respect to the metric $g$. Note that with this notation we have $\nabla^g H=\nabla^g dH$ for any smooth function $H \colon M \to \R$.
\end{notation}

\begin{proposition}\label{proposition ps}
Let $(\rho,g)$ be a Morse-Smale pair. Then for every $\delta>0$ there exists a constant $C_\rho=C_\rho(\delta,g)>0$ such that
\begin{equation*}
\Vert \nabla^g \rho (z) \Vert < C_\rho \implies \exists x \in Z(\rho) \text{ with } d(x,z)<\delta,
\end{equation*}
where both $\Vert \cdot \Vert$ and $d( \, \cdot \, , \, \cdot \, )$ are induced by $g$.
\end{proposition}

\begin{proof}
Suppose the assertion does not hold. Then there exists a $\delta>0$, a positive sequence $C_k \to 0$ and $(z_k) \subset M$ such that 
$$\Vert \nabla^g \rho(z_k) \Vert < C_k \text{ and } z_k \in M \setminus \bigcup_{x \in Z(\rho)}B_\delta(x).$$
By compactness of $M$ we can pass to a subsequence $(z_k)$ converging to some $z \in M$. The above however implies $\Vert \nabla^g \rho(z) \Vert=0$, which is equivalent to $z \in Z(\rho)$. At the same time $z$ lies in $M \setminus \bigcup_{x \in Z(\rho)}B_\delta(x)$, which is a contradiction. This concludes the proof.

\end{proof}
 
\begin{notation}
Such a constant $C_\rho>0$ is often referred to as a \emph{Palais-Smale constant} (short: \emph{PS-constant}). The main case of interest is the exact one, i.e. $\rho=dH$, for which we will abbreviate $C_{dH}=C_H$. Sometimes we will also abbreviate $C_{\rho}=C$.
\end{notation}

The next Lemma builds the main technical tool of Subsection \ref{subsection Section Perturbations}. The idea is to perturb one-forms $\alpha$ close to a given \emph{reference Morse-Smale pair} $(\rho,g)$ so that the pertubations, say $\alpha'$, maintain their cohomology classes of $\alpha$, become Morse, have the same zeros as $\rho$, and are still relatively close to $\rho$. This is reminiscent of Zhang's arguments \cite[Section 3]{zhang2019}.

\begin{lemma}\label{lemma perturbations}
Let $(\rho,g)$ be a Morse-Smale pair, $\delta>0$ so small that the balls $B_{2\delta}(x)$, with $x \in Z(\rho)$, are geodisically convex\footnote{This is a well known result in Riemannian geometry, see \cite{Whitehead1962}} and lie in pairwise disjoint charts of $M$, and $C=C_\rho(\delta,g)>0$ as in Proposition \ref{proposition ps}.

Let $\alpha \in \Omega^1(M)$ with 
\begin{equation*}
\Vert  \alpha - \rho \Vert < \frac{C}{8} \text{ and } a=[\alpha],
\end{equation*}where  $\Vert \cdot \Vert$ is the norm induced by $g$. Then there exists a Morse-Smale pair $(\alpha',g')$, with $\alpha' \in a$, satisfying 
\begin{itemize}
\item $\Vert \alpha' - \rho \Vert \leq 5 \cdot \Vert \alpha -\rho \Vert$, 
\item $\restr{\alpha'}{B_\delta(x)}=\restr{\rho}{B_\delta(x)}$ for all $x \in Z(\rho)$ and
\item $Z(\alpha')=Z(\rho)$.
\end{itemize}
Moreover $\Vert \nabla^{g'} \alpha'(z) \Vert' < \frac{C}{8}$ implies $z \in B_\delta(x)$\footnote{This is still the ball of radius $\delta$ with respect to the distance metric induced by $g$.} for some zero $x \in Z(\alpha')$, where $\Vert \cdot \Vert'$ is the norm induced by $g'$.

\end{lemma}

\begin{proof}
Since $\rho$ is a Morse form, there are only finitely many zeros $x \in Z(\rho)$. Around each such $x$ we will perturb $\alpha$ without changing its cohomology class: Enumerate the finitely many zeros of $\rho$ by $\{x_i\}_{i=1,\dots,k}$ and pick a bump functions $h_i \colon M \to \R$ with
$$\begin{cases}
h_i \equiv 0, &\text{ on } M\setminus B_{2\delta}(x_i) \\
h_i \equiv 1, &\text{ on } B_\delta(x_i) \\
\Vert \nabla^g h_i \Vert \leq \frac{2}{\delta}.
\end{cases}
$$
Since every $B_{2\delta}(x_i)$ is simply connected, there exist unique smooth functions $f_i \colon B_{2\delta}(x_i) \to \R$ satisfying  
\begin{equation}\label{equation unique primities in technical lemma}
f_i(x_i)=0 \text{ and } df_i=\restr{(\alpha-\rho)}{B_{2\delta}(x_i)}.
\end{equation} We set
\begin{equation}\label{equation perturbation in technical lemma}
\alpha'=\alpha- \sum_{i=1}^k d(h_i \cdot f_i).
\end{equation}
By construction we have $\alpha' \in a$, $\alpha'=\rho$ on $B_\delta(x)$ for $x \in Z(\rho)$, and that $\alpha'$ agrees with $\rho$ outisde of $\bigcup_{i=0}^k B_{2\delta}(x_i)$. Consequently $\alpha'-\rho$ and $\alpha-\rho$ agree outside of $\bigcup_{i=0}^k B_{2\delta}(x_i)$. This means that for the inequality in the first bullet point it suffices to argue why the bound holds inside each ball $B_{2 \delta}(x_i)$. Inserting the definitions grants
\begin{align*}
\Vert \alpha' -\rho \Vert_{B_{2\delta}(x_i)} &= \Vert \alpha- \rho - h_i \cdot df_i - f_i \cdot dh_i \Vert_{B_{2\delta}(x_i)} \\
&\leq (1-h_i) \Vert \alpha-\rho \Vert_{B_{2\delta}(x_i)}+\Vert f_i \Vert_{B_{2\delta}(x_i)} \cdot \Vert \nabla^g h_i \Vert_{B_{2\delta}(x_i)} \\
&\leq  \Vert \alpha-\rho \Vert_{B_{2\delta}(x_i)} + \frac{2}{\delta} \cdot \Vert f_i \Vert_{B_{2\delta}(x_i)}
\end{align*}
Recall that $f_i$ was chosen such that $f_i(x_i)=0$. Due to the geodesic convexity of the balls $B_{2\delta}(x_i)$ we can apply the mean value inequality
$$\vert f_i(y) \vert=\vert f_i(x_i)-f_i(y) \vert \leq \Vert \nabla^g f_i \Vert_{B_{2\delta}(x_i)} \cdot d(x,y) \leq \Vert \alpha -\rho \Vert_{B_{2\delta}(x_i)} \cdot 2\delta, \quad \forall y \in B_{2\delta}(x_i).$$
All in all this implies
$$\Vert \alpha'-\rho \Vert_{B_{2\delta}(x_i)} \leq \Vert \alpha- \rho \Vert_{B_{2\delta}(x_i)}+ \frac{4\delta}{\delta} \Vert \alpha-\rho \Vert_{B_{2\delta}(x_i)} =5 \cdot \Vert \alpha-\rho \Vert_{B_{2\delta}(x_i)}.$$
This proves the first inequality in the first bullet point. 
From this we will deduce that $Z(\rho)=Z(\alpha')$: the inclusion $Z(\rho)\subseteq Z(\alpha')$ is clear as $\rho$ agrees with $\alpha'$ around $Z(\rho)$. The reverse inclusion is obtained by observing that for $y \in Z(\alpha')$ we have
$$\Vert \nabla^g \rho(y) \Vert = \Vert \nabla^g \rho(y)-\underbrace{\nabla^g\alpha'(y)}_{=0} \Vert \leq \Vert \rho - \alpha' \Vert \leq 5 \cdot \Vert \alpha - \rho \Vert < C,$$ by assumption on $\rho$ and the inequality above. Proposition \ref{proposition ps} then implies that $z$ has to be a zero of $\rho$ as well. This proves $Z(\rho)=Z(\alpha')$, in particular that $\alpha'$ is a Morse form.

To get a Riemannian metric $g'$ that turns $(\alpha',g')$ into a Morse-Smale pair it suffices to perturb $g$ on an open set that intersects all the Novikov trajectories of $(\alpha',g)$, see \cite[Page 38-40]{Pozniak1999} for more details. Since $Z(\rho)=Z(\alpha')$, we can take a perturbation $g'$ that agrees with $g$ on $M \setminus \bigcup_{i=1}^k B_\delta(x_i)$ and is close to $g$ in the $C^{\infty}$-topology.

The last assertion of the statement follows from the observation that, for $\alpha'$ fixed, the map $g' \mapsto \Vert \nabla^{g'} \alpha' \Vert'$ is continuous, thus for $g'$ close to $g$ we get
$$\Vert \nabla^g \alpha'(z) \Vert \leq  \underbrace{ \vert \Vert \nabla^g \alpha'(z) \Vert- \Vert \nabla^{g'} \alpha'(z) \Vert' \vert}_{ < \varepsilon} + \Vert \nabla^{g'} \alpha'(z) \Vert'.$$ Assuming $\Vert \nabla^{g'}\alpha'(z) \Vert' < \frac{C}{8}$ we thus end up with
\begin{align*}
\Vert \nabla^g \rho (z) \Vert &\leq \Vert \nabla^g \rho(z) - \nabla^g \alpha'(z) \Vert + \Vert \nabla^g \alpha'(z) \Vert  \\
&\leq \Vert \rho-\alpha' \Vert + \varepsilon + \frac{C}{8} && \text{ using the above inequality}, \\
&\leq 4 \cdot \Vert \rho - \alpha \Vert + \varepsilon + \frac{C}{8} && \text{ by the first bullet point}, \\
&< \frac{4 \cdot C}{8}+ \varepsilon +\frac{C}{8} && \text{ by assumption.}
\end{align*}
Taking $\varepsilon \leq \frac{C}{4}$ and invoking Proposition \ref{proposition ps} then concludes the proof.

\end{proof}

Lemma \ref{lemma perturbations} can be applied to a whole section $\theta \colon \mathcal{A} \to \Omega^1(M)$ \emph{nearby} a reference Morse-Smale pair $(\rho,g)$ and give rise to a perturbed section $\vartheta \colon \mathcal{A} \to \Omega^1(M)$ that is still relatively close to $\rho$, so that each $\vartheta_a$ agrees with $\rho$ near the zeros $x \in Z(\rho)$.

\begin{proposition}\label{proposition perturbation for the whole section of polytope}
Let $(\rho,g)$ and $C=C_{\rho}>0$ as in Lemma \ref{lemma perturbations}, denote $N=\bigcup_{i} B_\delta(x_i)$ with $x_i \in Z(\rho)$, and let $\theta \colon \mathcal{A} \to \Omega^1(M)$ be a section satisfying
\begin{equation}\label{equation section sigma is close to dH}
\Vert \theta_a-\rho \Vert < \frac{C}{8}.
\end{equation}
Then there exists a section 
\begin{equation*}
\vartheta=\vartheta(\theta,\rho,g) \colon \mathcal{A} \longrightarrow \Omega^1(M)
\end{equation*} 
and a positive constant $D=D(N,g)>0$ with the following significance:\footnote{The choice of $D>0$ is independent of the assumption \eqref{equation section sigma is close to dH}.}
\begin{itemize}
\item $Z(\vartheta_a)=Z(\rho)$ for all $a \in \mathcal{A}$,
\item $\restr{\vartheta_a}{B_\delta(x_i)}=\restr{\rho}{B_\delta(x_i)}$ for all $x_i \in Z(\rho), \, a \in \mathcal{A}$.
\end{itemize}
Moreover, for every $\vartheta_a$ there exists a Riemannian metric $g_{\vartheta_a}$ close to $g$ with $\restr{g_{\vartheta_a}}{M\setminus N}=\restr{g}{M\setminus N}$ such that 
\begin{itemize}
\item $(\vartheta_a,g_{\vartheta_a})$ is Morse-Smale,
\item $\Vert \vartheta_b-\rho \Vert_{\vartheta_a} \leq D \cdot \Vert \vartheta_b- \rho \Vert \leq 5 \cdot D \cdot \Vert \theta_b-\rho \Vert$ for all $a,b \in \mathcal{A},$
\end{itemize}
where $\Vert \cdot \Vert_{\vartheta_a}$ is the operator norm induced by $g_{\vartheta_a}$.

\end{proposition}

\begin{proof}
Since the whole section $\theta \colon \mathcal{A} \to \Omega^1(M)$ is $\frac{C}{8}$-close to $(\rho,g)$, we can take $(\rho,g)$ as a reference pair and apply Lemma \ref{lemma perturbations} to every $\theta_a$ and denote $\vartheta_a$ the corresponding perturbation. Recall from the proof of Lemma \ref{lemma perturbations} that $\vartheta_a$ is obtained by an exact perturbation of $\theta_a$ around the zeros of $\rho$ -- a closer inspection reveals that this exact perturbation varies smoothly along $\theta_a$, in particular that $\vartheta$ defines a smooth section. The first two bullet points follow immediately from Lemma \ref{lemma perturbations}.

We choose $g_{\vartheta_a}=g_{(\theta_a)'}$ just as $g'$ in Lemma \ref{lemma perturbations}, i.e. by means of a small perturbation of $g$ inside $N$. The argument in \cite{Pozniak1999} shows that sufficiently small perturbations give rise to Riemannian metrics that are \emph{uniformly} equivalent to the original $g$, in other words we may choose $g_{\vartheta_a}$ such that $(\vartheta_a,g_{\vartheta_a})$ is Morse-Smale and
$$\frac{1}{D^2} g_{\vartheta_a}(v,v) \leq g(v,v) \leq D^2g_{\vartheta_a}(v,v), \quad \forall v \in TM, \, \forall a \in \mathcal{A},$$ with $D>0$ a constant that only depends on $N$ and $g$. Using this inequality and invoking the first bullet point of Lemma \ref{lemma perturbations} concludes the proof.

\end{proof}
\subsection{Section perturbations}\label{subsection Section Perturbations}

We can finally state and prove Theorem \ref{theorem perturbed section} by applying the previous results in the special case of \emph{exact} reference pairs:

\begin{theorem}\label{theorem perturbed section}
Let $\theta \colon \mathcal{A} \to \Omega^1(M)$ be a section and a reference Morse-Smale pair $(H,g)$ on $M$. Then there exists a perturbed section 
\begin{equation*}
\vartheta=\vartheta(\theta,H,g) \colon \mathcal{A} \longrightarrow \Omega^1(M)
\end{equation*} and a choice of Riemannian metrics $g_{\vartheta_a}$
with the following significance:
\begin{itemize}
\item \emph{\textbf{(Morse-Smale property)}} Each pair $(\vartheta_a,g_{\vartheta_a})$ is Morse-Smale, for all $a \in \mathcal{A}$,
\item \emph{\textbf{(Chain-complex)}} The chain complex $\left(\mathrm{CN}_\bullet\left(\vartheta_a,g_{\vartheta_a},\mathcal{A} \right),\partial_{\vartheta_a} \right)$ is well defined for every pair $(\vartheta_a,g_{\vartheta_a})$ as above,
\item \emph{\textbf{(Ray-invariance)}} The chain complexes are equal upon scaling, i.e. $\mathrm{CN}_\bullet\left(\vartheta_a,g_{\vartheta_a},\mathcal{A} \right)=\mathrm{CN}_\bullet\left(r \cdot \vartheta_a, g_{\vartheta_a}, r \cdot \mathcal{A} \right)$ for all $r>0$, $a \in \mathcal{A}$.\footnote{Note that here the metric is \emph{not} scaled.}
\end{itemize}
\end{theorem}

The rough idea is to \emph{``shift-and-scale"}: we shift and scale the polytope $\mathcal{A}$ so that it is sufficiently close to a given exact one-form $dH$ in the operator norm $\Vert \cdot \Vert$ coming from $g$. Then one can perturb the scaled section by means of Proposition \ref{proposition perturbation for the whole section of polytope} and scale back. This will be the desired section $\vartheta$ on $\mathcal{A}$. By construction we will then see that the three bullet points are satisfied. The choices involved (i.e. choice of section $\theta$, reference pair $(H,g)$ and perturbation coming from Theorem \ref{theorem perturbed section}) will prove harmless -- they result in chain homotopy equivalent complexes. This is proven in the next subsection (cf. Theorem \ref{theorem independence of chain complexes}). 

At the cost of imposing a smallness condition on the underlying section, we get the same results for perturbations associated to non-exact reference pairs (cf. Corollary \ref{corollary other chain complex perturbed}) and the same independence of auxiliary data holds (cf. Theorem \ref{theorem indep of perturb with exact vs non-exact ref. pair}).

\begin{proof}[Proof of Theorem \ref{theorem perturbed section}]
As a first candidate for $\vartheta$, we pick $$\vartheta \colon \mathcal{A} \longrightarrow \Omega^1(M), \quad \vartheta_a:= \theta_a + dH.$$ This is still a section, but does not satisfy the bullet points above. Since $\theta$ is smooth, there exists $$\varepsilon=\varepsilon(\theta,H,g,\delta)>0,$$ such that 
$$\varepsilon \cdot \theta_a+dH \text{ is } \frac{C_H}{D \cdot 1000}\text{-close to } dH,$$ with respect to $\Vert \cdot \Vert$ induced by $g$, $C_H=C_H(\delta,g)>0$ and $D=D(N,g)>0$ chosen as in Proposition \ref{proposition perturbation for the whole section of polytope}. Now we can apply Proposition \ref{proposition perturbation for the whole section of polytope} to the section $ \varepsilon \cdot a \mapsto \varepsilon \cdot \theta_a+dH$ and obtain a new section
$$ \vartheta^{\varepsilon} \colon \varepsilon \cdot \mathcal{A} \longrightarrow \Omega^1(M). \footnote{Explicitly, this section is of the form $$\varepsilon \cdot a \mapsto \varepsilon \cdot \theta_a+dH + \sum_i d(f_i \cdot h_i),$$ where $f_i$ depends smoothly on $\theta_a$ and satisfies $df_i=\varepsilon \cdot \theta_a, \, f_i(x_i)=0$ around critical points $x_i$ of $H$, see Lemma \ref{lemma perturbations}, \eqref{equation unique primities in technical lemma} and \eqref{equation perturbation in technical lemma} applied to $\varepsilon \cdot \theta_a+ dH$ and $\rho=dH$.}.$$ Finally we scale back and redefine
\begin{equation*}
\vartheta=\vartheta(\theta,H,g,\varepsilon) \colon \mathcal{A} \longrightarrow \Omega^1(M), \quad a \mapsto \frac{1}{\varepsilon} \cdot \vartheta^{\varepsilon}(\varepsilon \cdot a).
\end{equation*}
Thus we have $$\varepsilon \cdot \vartheta_a=\vartheta^{\varepsilon}(\varepsilon \cdot a), \quad \forall a \in \mathcal{A}.$$
For each $\vartheta^{\varepsilon}(\varepsilon \cdot a)$ we choose a Riemannian metric denoted by $g_{a}$ as in Proposition \ref{proposition perturbation for the whole section of polytope}. Thus $(\vartheta^{\varepsilon}(\varepsilon \cdot a),g_{a})$ is Morse-Smale, and so is $(\vartheta_a,g_{a})$ since scaling does not affect the Morse-Smale property. This proves the first bullet point.

\begin{claim no number}\label{claim 1 in proof of theorem}
$\mathrm{CN}_\bullet\left(\vartheta^{\varepsilon}(\varepsilon \cdot a), \varepsilon \cdot \mathcal{A} \right)=\mathrm{CN}_\bullet(\varepsilon \cdot \vartheta_a,\varepsilon \cdot \mathcal{A})$ is a well-defined chain complex for any $a \in \mathcal{A}$.
\end{claim no number}
Indeed, assume for contradiction that there exists a Novikov-chain $\xi=\sum_{\tilde{x}} \xi_{\tilde{x}} \, \tilde{x} \in \mathrm{CN}_\bullet\left(\vartheta^{\varepsilon}(\varepsilon \cdot a), \varepsilon \cdot \mathcal{A} \right)$ such that $$\partial \xi \notin \mathrm{CN}_\bullet\left(\vartheta^{\varepsilon}(\varepsilon \cdot a), \varepsilon \cdot \mathcal{A} \right).$$ This means that there is some $\varepsilon \cdot b \in \varepsilon \cdot \mathcal{A}$, $c \in \R$ and sequences $\tilde{x}_n$ with $\xi_{\tilde{x}_n} \neq 0$, $\tilde{y}_n$ pairwise distinct, $\tilde{\gamma}_n \in \mathcal{\underline{M}}\left(\tilde{x}_n,\tilde{y}_n;\tilde{f}_{\varepsilon \cdot \vartheta_a}\right)$, and $$\tilde{f}_{\vartheta^{\varepsilon}(\varepsilon \cdot b)}(\tilde{y}_n)=\tilde{f}_{\varepsilon \cdot \vartheta_b}(\tilde{y}_n) \geq c,$$ see Remark \ref{remark indep of primitive and cohomol for finiteness condition}. 
Denote by $\gamma_n=\pi \circ \tilde{\gamma}_n$ the Novikov trajectories downstairs. The energy expression can then be massaged as follows:
\begin{align*}
0 \leq E(\tilde{\gamma}_n)=E(\gamma_n) &=-\int_{\gamma_n} \vartheta^{\varepsilon}(\varepsilon \cdot a) \\
&=-\int_{\gamma_n} \vartheta^{\varepsilon}(\varepsilon \cdot b) + \int_{\gamma_n} \vartheta^{\varepsilon}(\varepsilon \cdot b)-\vartheta^{\varepsilon}(\varepsilon \cdot a) \\
&=\tilde{f}_{\vartheta^{\varepsilon}(\varepsilon \cdot b)}(\tilde{x}_n)-\tilde{f}_{\vartheta^{\varepsilon}(\varepsilon \cdot b)}(\tilde{y}_n)+\int_{\gamma_n} \vartheta^{\varepsilon}(\varepsilon \cdot b)-\vartheta^{\varepsilon}(\varepsilon \cdot a) \\
&\leq \tilde{f}_{\vartheta^{\varepsilon}(\varepsilon \cdot b)}(\tilde{x}_n)-c+\int_{\gamma_n} \vartheta^{\varepsilon}(\varepsilon \cdot b)-\vartheta^{\varepsilon}(\varepsilon \cdot a).
\end{align*}
Showing that the rightmost term is bounded by $m \cdot E(\gamma_n), \, m \in (0,1)$ suffices to obtain a contradiction: admitting such a bound leads to 
$$0 \leq E(\gamma_n) \leq (1-m)^{-1} \cdot \left( \tilde{f}_{\vartheta^{\varepsilon}(\varepsilon \cdot b)}(\tilde{x}_n) - c \right).$$ In particular, $c \leq \tilde{f}_{\vartheta^{\varepsilon}(\varepsilon \cdot b)}(\tilde{x}_n)$ for all $n$. But $\xi$ belongs to $ \mathrm{CN}_\bullet\left(\varepsilon \cdot \vartheta_a, \varepsilon \cdot \mathcal{A} \right)$ and $\xi_{\tilde{x}_n} \neq 0$, thus the multi finiteness condition implies that there are only finitely many distinct $\tilde{x}_n$. Up to passing to a subsequence we can therefore assume $\tilde{x}_n=\tilde{x}$ and also $\tilde{y}_n \in \pi^{-1}(y)$.\footnote{The latter is possible since $Z(\varepsilon \cdot \vartheta_a)=Z(\varepsilon \cdot dH)=\mathrm{Crit}(H)$ is finite.} The corresponding Novikov trajectories $$\gamma_n \in \mathcal{\underline{M}}\left(x,y;\vartheta^{\varepsilon}(\varepsilon \cdot a)\right)$$ have uniformly bounded energy
$$E(\gamma_n) \leq (1-m)^{-1} \cdot  \left(\tilde{f}_{\vartheta^{\varepsilon}(\varepsilon \cdot b)}(\tilde{x})-c \right),$$ therefore $\gamma_n$ has a $C^{\infty}_{\mathrm{loc}}$-convergent subsequence. At the same time $\mathcal{\underline{M}}\left(x,y;\vartheta^{\varepsilon}(\varepsilon \cdot a)\right)$ is a $0$-dimensional manifold, which means that the convergent subsequence $\gamma_n$ eventually does not depend on $n$. This contradicts our assumption that the endpoints $\tilde{y}_n$ upstairs are pairwise disjoint.

Therefore we are only left to show the bound
$$A(\gamma_n):=\int_{\gamma_n} \vartheta^{\varepsilon}(\varepsilon \cdot b)-\vartheta^{\varepsilon}(\varepsilon \cdot a) \leq \frac{1}{2}E(\gamma_n)$$ to conclude the Claim. For this purpose we define $$\mathcal{S}_n:= \left\{ s \in \R \, \big | \, \Vert \nabla^{g_{a}} \left(\vartheta^{\varepsilon}(\varepsilon \cdot a) \right)(\gamma_n(s)) \Vert_{g_{a}} \geq \frac{C_H}{8} \right\}.$$
The crucial observation is that both $\vartheta^{\varepsilon}(\varepsilon \cdot b)$ and $\vartheta^{\varepsilon}(\varepsilon \cdot a)$ agree with $dH$ around $\mathrm{Crit}(H)$, by choice of $\vartheta^{\varepsilon}$ via Proposition \ref{proposition perturbation for the whole section of polytope}. In particular
$$\restr{\vartheta^{\varepsilon}(\varepsilon \cdot b)-\vartheta^{\varepsilon}(\varepsilon \cdot a)}{B_\delta(z)}=0, \quad \forall z \in \mathrm{Crit}(H)=Z(\vartheta^{\varepsilon}(\varepsilon \cdot b))=Z(\vartheta^{\varepsilon}(\varepsilon \cdot a)).$$
Lemma \ref{lemma perturbations} says that for $s \in \R \setminus \mathcal{S}_n$ we get $$\gamma_n(s) \in \bigcup_{z \in Z(\vartheta^{\varepsilon}(\varepsilon \cdot a))}B_\delta(z),$$ consequently
\begin{equation}\label{equation no contribution around zeros}
\int_{\R \setminus \mathcal{S}_n} \left( \vartheta^{\varepsilon}(\varepsilon \cdot b)-\vartheta^{\varepsilon}(\varepsilon \cdot a)\right) \, \dot{\gamma}_n(s) \, ds=0.
\end{equation}
The Lebesgue measure $\mu(\mathcal{S}_n)$ can be bounded using the energy:
$$E(\gamma_n)=-\int_{\gamma_n} \vartheta^{\varepsilon}(\varepsilon \cdot a) =\int_\R \Vert \nabla^{g_{a}} \left(\vartheta^{\varepsilon}(\varepsilon \cdot a) \right)(\gamma_n(s)) \Vert_{g_{a}}^2 \, ds \geq \mu(\mathcal{S}_n) \cdot \left(\frac{C_H}{8} \right)^2,$$
thus 
\begin{equation}\label{equation lebesgue measure outside zero nbhd bound}
\mu(\mathcal{S}_n) \leq \left(\frac{8}{C_H} \right)^2 \cdot E(\gamma_n).
\end{equation}
And finally
\begin{align*}
\vert A(\gamma_n) \vert &\leq \Vert \vartheta^{\varepsilon}(\varepsilon \cdot b)-\vartheta^{\varepsilon}(\varepsilon \cdot a) \Vert_{g_{a}} \cdot \int_{\mathcal{S}_n} \Vert \dot{\gamma}_n(s) \Vert_{g_{a}} \, ds &&\text{ by } \eqref{equation no contribution around zeros}, \\
&\leq \big (\Vert \vartheta^{\varepsilon}(\varepsilon \cdot b)- dH \Vert_{g_{a}}+ \\
& \quad \quad + \Vert dH-\vartheta^{\varepsilon}(\varepsilon \cdot a) \Vert_{g_{a}} \big) \cdot \mu(\mathcal{S}_n)^{\frac{1}{2}} \cdot E(\gamma_n)^{\frac{1}{2}} \\
&\leq 5 D \cdot \big(\Vert \varepsilon \cdot \vartheta_b- dH \Vert +  \\
& \quad \quad +\Vert dH-\varepsilon \cdot \vartheta_a \Vert \big) \cdot \frac{8}{C_H} \cdot E(\gamma_n) &&\text{ Proposition } \ref{proposition perturbation for the whole section of polytope}, \, \eqref{equation lebesgue measure outside zero nbhd bound}, \\
&\leq \frac{80 \cdot D \cdot  C_H}{D \cdot 1000 \cdot C_H} \cdot E(\gamma_n) && \text{ by choice of scaling } \varepsilon>0, \\
&<\frac{1}{10}E(\gamma_n).
\end{align*}
This proves the Claim.

Now we observe that scaling $\vartheta_a$ by $r>0$ does not affect the zeros and that the moduli spaces associated to $(\vartheta_a,g_{a})$ are in one to one correspondence with those of $(r \cdot \vartheta_a,g_{a})$. It is also clear that the multi finiteness condition imposed by $\mathcal{A}$ is equivalent to that of $r \cdot \mathcal{A}$. All in all this means that for any $r>0$ the polytope chain complexes associated to $(r \cdot \vartheta_a,g_{a})$ agree with each other. This proves the ray-invariance. Setting $r=\frac{1}{\varepsilon}$ and using the Claim proves the remaining first bullet point.

\end{proof}

\begin{remark}\label{remark scaling of polytope is the same as scaling of reference Morse function}
Instead of running the argument for the sections $\vartheta^{\varepsilon} \colon \varepsilon \cdot \mathcal{A} \to \Omega^1(M)$ we could also work with $\vartheta=\frac{1}{\varepsilon} \cdot \vartheta^{\varepsilon} \colon \mathcal{A} \to \Omega^1(M)$ by directly by applying Proposition \ref{proposition perturbation for the whole section of polytope} to the section $$a \mapsto \theta_a+ d\left( \varepsilon^{-1}H\right)$$ and $(\varepsilon^{-1}H,g)$. These two approaches are equivalent, the only difference is psychological: we find it more natural to visualize the shrinking of the polytope opposed to the scaling of Morse functions.
Note that the analogous bound at the end of the proof of Theorem \ref{theorem perturbed section} holds upon replacing $H$ by $\varepsilon^{-1}H$. This follows from the nice scaling behaviour of the PS-constants: $$C_{\varepsilon^{-1}H}(\delta,g)=\varepsilon^{-1}C_H(\delta,g),$$ therefore  ``$\varepsilon \cdot \theta_a+dH$ is $\frac{C_H}{D \cdot 1000}$-close to $dH$" if and only if ``$\theta_a + d(\varepsilon^{-1}H)$ is $\frac{C_{\varepsilon^{-1}H}}{D \cdot 1000}=\frac{C_H}{\varepsilon \cdot D \cdot 1000}$-close to $d(\varepsilon^{-1}H)$". 

\end{remark}

\begin{remark}
The skeptical reader might wonder whether $\partial_{\vartheta_a}^2=0$ really holds. Viewing the chain group $\mathrm{CN}_\bullet(\vartheta_a,\mathcal{A})$ as a certain twisted chain group allows for a quick and simple proof -- see Remark \ref{remark novikov weights to prove partial squared equal to 0}.
\end{remark}

The key in the proof of Theorem \ref{theorem perturbed section} was to obtain control over the energy by perturbing the section $\theta \colon \mathcal{A} \to \Omega^1(M)$ via Proposition \ref{proposition perturbation for the whole section of polytope}. The perturbation was chosen so that there would be no contribution to the energy near the zeros. Similar ideas to control the energy can be found in \cite[Subsection 3.6.2]{benedetti2016}, \cite{zhang2019}. 

The question remains why we used an (exact) reference pair $(H,g)$ instead of a more general Morse-Smale pair $(\rho,g)$ in Theorem \ref{theorem perturbed section}. The answer is simple: the given argument already breaks down in the very first line -- the corresponding $\vartheta$ is not a section anymore, since the $\rho$-shift changes the cohomology class. However, whenever the section $\theta \colon \mathcal{A} \to \Omega^1(M)$ is already \emph{sufficiently close} to $(\rho,g)$ in terms of the corresponding PS-constant $C_\rho>0$, we do not need to shift and scale $\theta$, and can perturb $\theta$ directly:
\begin{corollary}\label{corollary other chain complex perturbed}
Let $(\rho,g)$ be a Morse-Smale pair, $C=C_\rho>0$, $N=\bigcup_i B_\delta(x_i)$ with $x_i \in Z(\rho)$, and $D=d(N,\delta)>0$ as in Proposition \ref{proposition perturbation for the whole section of polytope}. Let $\theta \colon \mathcal{A} \to \Omega^1(M)$ be a smooth section such that
\begin{equation*}
\Vert \theta_a-\rho \Vert < \frac{C_\rho}{D \cdot 1000 },
\end{equation*}
with $\Vert \cdot \Vert$ the operator norm induced by $g$.
Then there exists a perturbed section
\begin{equation}
\vartheta=\vartheta(\theta,\rho,g) \colon \mathcal{A} \longrightarrow \Omega^1(M),
\end{equation}
and $g_{\vartheta_a}$ such that the same conclusions as in Theorem \ref{theorem perturbed section} hold.
\end{corollary}

\begin{proof}
Upon replacing $\varepsilon \cdot \theta_a+dH$ and $dH$ with $\theta_a$ and $\rho$, the proof is word for word the same as the one of Theorem \ref{theorem perturbed section}.

\end{proof}
 
\subsection{Independence of the data}

The section $\vartheta=\vartheta(\theta,H,g)$ constructed in Theorem \ref{theorem perturbed section} does not only depend on $(\theta,H,g)$, but also comes with a choice of scaling $\varepsilon(\theta,H,g,\delta)>0$. We shall prove that any valid perturbation $\vartheta^i=\vartheta^i(\theta_i,H_i,g_i,\varepsilon_i)$ in the sense of Theorem \ref{theorem perturbed section} gives rise to chain homotopy equivalent chain complexes. The same is true for perturbations coming from Corollary \ref{corollary other chain complex perturbed} and at the end of the subsection we will show that both perturbations lead to chain homotopy equivalent Novikov complexes.

\begin{remark}
All the chain maps and chain homotopy equivalences constructed from here on are Novikov-module morphisms, i.e. linear over the Novikov ring. We will not explicitly state this every time for better readability.
\end{remark}

\begin{theorem}\label{theorem independence of chain complexes}
For $i=0,1$, let $\theta_i \colon \mathcal{A} \to \Omega^1(M)$ be sections, $(H_i,g_i)$ Morse-Smale pairs and $\delta_i>0, \, \varepsilon_i(\theta_i,H_i,g_i,\delta_i)>0 $  as in proof of Theorem \ref{theorem perturbed section}. 

Then any two perturbed sections 
\begin{equation}
\vartheta^i=\vartheta^i(\theta_i,H_i,g_i,\varepsilon_i) \colon \mathcal{A} \to \Omega^1(M),
\end{equation} in the sense of Theorem \ref{theorem perturbed section}, induce chain homotopy equivalent polytope complexes:
\begin{equation}
\mathrm{CN}_\bullet(\vartheta^0_a,\mathcal{A}) \simeq \mathrm{CN}_\bullet(\vartheta^1_a,\mathcal{A}), \quad \forall a \in \mathcal{A}.
\end{equation}

\end{theorem}

\begin{proof}[Proof of Theorem \ref{theorem independence of chain complexes}]

We may assume $\varepsilon_1 \geq \varepsilon_0$. Denote by $$\vartheta^i(\sigma_i,H_i,g_i,\varepsilon_i), \quad i=0,1,$$ the respective sections on $\mathcal{A}$ as in the first part of the proof of Theorem \ref{theorem perturbed section}. Let 
\begin{equation}\label{equation smooth monotone function h}
h \colon [0,1] \to \R
\end{equation}
be a smooth function with $h \equiv 0$ on $(-\infty,e)$ and $h\equiv 1$ on $(1-e,+\infty)$, for some small $e>0$, and set
\begin{align*}
\vartheta^s &=(1-h(s)) \cdot \vartheta^0 +h(s) \cdot \vartheta^1.
\end{align*}
Fix $a \in \mathcal{A}$ and pick $g_i:=g_{\vartheta^i_a}$, $i=0,1$ two metrics as in Theorem \ref{theorem perturbed section}. Let $g_s=g_s(a)$ be a homotopy of Riemannian metrics connecting $g_0$ to $g_1$ and assume that $(\vartheta^s_a,g_s)$ is regular -- this is rectified by Remark \ref{remark homotopy regular no problemo} below. Note that here $g_s$ actually depends on $a$.

To this regular homotopy we can now associate a chain continuation 
\begin{equation}\label{equation chain continuation map tau s}
\Psi^{10} \colon \mathrm{CN}_\bullet(\vartheta^0_a,\mathcal{A}) \longrightarrow \mathrm{CN}_\bullet(\vartheta^1_a,\mathcal{A}), \quad \xi=\sum_{\tilde{x}} \xi_{\tilde{x}} \, \tilde{x} \mapsto \sum_{\tilde{x},\tilde{y}} \xi_{\tilde{x}} \cdot \#_{\mathrm{alg}} \, \mathcal{M}\left(\tilde{x},\tilde{y};\tilde{f}_{\vartheta^s_a}\right) \, \tilde{y}.
\end{equation}
Analogously to the case of the boundary operator in Theorem \ref{theorem perturbed section}, proving that $\Psi^{10}$ defines a well defined Novikov chain map essentially boils down to proving that it respects the multi finiteness condition -- the rest follows by standard Novikov-Morse techniques. Thus, proceeding as in Theorem \ref{theorem perturbed section} reveals that it suffices\footnote{This is also implicitly using that $\tilde{f}_{\vartheta^s_b}=\tilde{f}_{\vartheta^0_b}+h_s\circ \pi$ for a smooth family $h_s \in C^{\infty}(M)$ since $\vartheta^s_b$ are cohomologous for all $s$. Hence $$-\int_{\gamma_n} \vartheta^s_b=-\tilde{f}_{\vartheta^1_b}(\tilde{y}_n)+\tilde{f}_{\vartheta^0_b}(\tilde{x}_n)+\underbrace{\int_{[0,1]} \frac{\partial h_s}{\partial s}(\gamma_n(s)) \, ds}_{\leq C},$$ for some uniform constant $C$.} to bound 
\begin{equation}\label{equation bad guy in espilon s contiuaiton}
\int_{\gamma_n} \vartheta^s_b-\vartheta^s_a
\end{equation}
by either a multiple $m \in (0,1)$ of the energy $E(\gamma_n)$, where $b$ is some cohomology class in $\mathcal{A}$, or a uniform bound\footnote{Uniform in $b$ and $n \in \N$, that is.} altogether. Set 
\begin{align*}
\mathcal{S}_n^0&:=\ \left\lbrace s \in (-\infty,0) \, \big | \,\Vert \nabla^{g_0} \vartheta^0_a (\gamma_n(s)) \Vert_{g_0} \geq \frac{C_H}{\varepsilon_0 \cdot 8} \right\rbrace, \\ \mathcal{S}_n^1&:=\ \left\lbrace s \in (1,+\infty) \, \big | \,\Vert \nabla^{g_1} \vartheta^1_a (\gamma_n(s)) \Vert_{g_1} \geq \frac{C_H}{\varepsilon_1 \cdot 8} \right\rbrace.
\end{align*}
This time around we need to divide by $\varepsilon_i$ as we are running the continuation directly on the original polytope $\mathcal{A}$ instead of the scaled polytope (see Remark \ref{remark scaling of polytope is the same as scaling of reference Morse function}). As in the previous proof of Theorem \ref{theorem perturbed section}, the $s \in \R_{\leq 0}\setminus \mathcal{S}_n^0$ and $s \in \R_{\geq 1} \setminus \mathcal{S}_n^1$ do not contribute to \eqref{equation bad guy in espilon s contiuaiton} as $\gamma_n(s)$ will be near the zeros of $H_i$, where $\vartheta^i_b=\vartheta^i_a$. On the other hand, using similar arguments we obtain
\begin{align*}
\bigg \vert \int_{\mathcal{S}_n^0 \cup \mathcal{S}_n^1} \left(\vartheta^s_b- \vartheta^s_a\right) \, \dot{\gamma}_n(s) \, ds \bigg \vert &\leq 2\cdot \max_{i=0,1} \Vert \vartheta^i_b-\vartheta^i_a \Vert_i \cdot \mu(\mathcal{S}^i_n)^{\frac{1}{2}} \cdot E(\gamma_n)^{\frac{1}{2}}   \\
&\leq 2\cdot \max_{i=0,1} \left(\Vert \vartheta^i_b-dH_i \Vert_i+ \Vert dH_i - \vartheta^i_a \Vert_i \right) \cdot \frac{\varepsilon_i \cdot 8}{C_H} \cdot E(\gamma_n)\\
&\leq \max_{i=0,1}  \frac{ 4 \cdot 5 \cdot D \cdot \varepsilon_i \cdot 8 \cdot C_H }{D \cdot \varepsilon_i \cdot 1000 \cdot C_H} \cdot E(\gamma_n) \\
&\leq \frac{1}{5} \cdot E(\gamma_n),
\end{align*}
where we have used Proposition \ref{proposition perturbation for the whole section of polytope} as in Theorem \ref{theorem perturbed section}.
We are left to bound \eqref{equation bad guy in espilon s contiuaiton} for $s \in [0,1]$. For this we compute via Cauchy-Schwarz:
\begin{align*}
\bigg \vert \int_{[0,1]} \left(\vartheta^s_b- \vartheta^s_a\right) \, \dot{\gamma}_n(s) \, ds \bigg \vert &\leq \max_{s \in [0,1]} \Vert \vartheta^s_b-\vartheta^s_a \Vert_s \cdot E(\gamma_n)^{\frac{1}{2}}.
\end{align*}
By compactness of $[0,1], \, \mathcal{A}$ and continuity of $\vartheta \colon [0,1] \times \mathcal{A} \to \Omega^1(M)$, we may bound
$$\max_{s \in [0,1]} \Vert \vartheta^s_b-\vartheta^s_a \Vert_s \leq F,$$ where $F>0$ is a uniform constant in $s \in [0,1]$ and $b \in \mathcal{A}$ -- recall that $g_s$ depends on $a$, but that does not matter. In particular, this proves
\begin{equation*}
\bigg \vert \int_{\gamma_n} \vartheta^s_b-\vartheta^s_a \, \bigg \vert \leq F \cdot E(\gamma_n)^{\frac{1}{2}}+\frac{1}{5} \cdot E(\gamma_n), \quad \forall n \in \N, b \in \mathcal{A}.
\end{equation*}
A case distinction now does the job: for any $n \in \N$ we either have $\frac{1}{5}E(\gamma_n) \geq F \cdot E(\gamma_n)^{\frac{1}{2}}$ or $\frac{1}{5}E(\gamma_n) < F \cdot E(\gamma_n)^{\frac{1}{2}}$. In the first case we can bound the norm of \eqref{equation bad guy in espilon s contiuaiton} by $\frac{2}{5} \cdot E(\gamma_n)$, whereas in the second case we get $\frac{1}{5} \cdot E(\gamma_n)^{\frac{1}{2}}<F$ and thus we may bound the norm of \eqref{equation bad guy in espilon s contiuaiton} by $10F^2$. This proves
$$\bigg \vert \int_{\gamma_n} \vartheta^s_b-\vartheta^s_a \bigg \vert \leq \max \left\lbrace 10F^2, \frac{2}{5}E(\gamma_n) \right\rbrace.$$ As explained before, this suffices to conclude that $\Psi^{10}$ defines a well defined Novikov chain map, which defines the desired chain homotopy equivalence (see proof of Theorem \ref{theorem perturbed section} for more details). This concludes the proof.

\end{proof}

\begin{remark}\label{remark homotopy regular no problemo}
The (linear) homotopy $(\vartheta^s_a,g_s)$ chosen in the proof of Theorem \ref{theorem independence of chain complexes} might be non-regular. One can replace $(\vartheta^s_a,g_s)$ with an arbitrarily close regular homotopy $((\vartheta^s_a)',g_s')$ connecting the same data. The only bit where this affects the previous argument in Theorem \ref{theorem independence of chain complexes} is when trying to bound $\max_{s \in [0,1]} \Vert \vartheta^s_b-(\vartheta^s_a)' \Vert_s'$. By using that $\vartheta^s_a$ is smooth in $s$ and close to $(\vartheta^s_a)'$, we still get the desired uniform bound $b$.
\end{remark}

As a consequence of Theorem \ref{theorem perturbed section} and Theorem \ref{theorem independence of chain complexes} we obtain the analogue results for restrictions to subpolytopes $\mathcal{B}\subseteq \mathcal{A}$:

\begin{corollary}\label{corollary restricted polytopes also give well defined chain complexes}
Let $\theta \colon \mathcal{A} \to \Omega^1(M)$ be a section. Then for every perturbed section $\vartheta \colon \mathcal{A} \to \Omega^1(M)$ coming from Theorem \ref{theorem perturbed section} and subpolytope $\mathcal{B} \subseteq \mathcal{A}$ we obtain a well defined polytope chain complex 
\begin{equation}\label{equation inclusion chain complexes upon restricting}
\left(\mathrm{CN}_\bullet \left( \vartheta_a,\mathcal{A}  |_{\mathcal{B}}\right),\partial_{\vartheta_a}\right), \quad \forall a \in \mathcal{A},
\end{equation}
satisfying all the bullet points of Theorem \ref{theorem perturbed section}. Any other choice $\vartheta'=\vartheta'(\theta',H',g')$ does not affect the chain complexes up to chain homotopy equivalence.

Moreover, the inclusion 
\begin{equation}\label{equation chain inclusions of restrictions}
\iota_{\mathcal{B}} \colon \mathrm{CN}_\bullet(\vartheta_a,\mathcal{A}) \longrightarrow \mathrm{CN}_\bullet\left(\vartheta_a,\mathcal{A}  |_{\mathcal{B}}\right)
\end{equation}
defines a Novikov-linear chain map for all $a \in \mathcal{A}$.
\end{corollary}

\begin{proof}
The proof of the first part is literally the same as in Theorem \ref{theorem perturbed section} and Theorem \ref{theorem independence of chain complexes}. To see that the inclusion defines a chain map it suffices to observe that both boundary operators in \eqref{equation chain inclusions of restrictions} are identical upon restricting to the smaller complex $\mathrm{CN}_\bullet(\vartheta_a,\mathcal{A})$.

\end{proof}

In the preceding subsection we also defined a polytope chain complex variant using perturbed sections with respect to non-exact reference pairs (cf. Corollary \ref{corollary other chain complex perturbed}). While this variant requires the underlying section to satisfy some a priori smallness conditions, it does agree with the polytope chain complex variant of Theorem \ref{theorem perturbed section}.

\begin{theorem}\label{theorem indep of perturb with exact vs non-exact ref. pair}
Let $(\rho,g)$ and $\theta \colon \mathcal{A} \to \Omega^1(M)$ be as in Corollary \ref{corollary other chain complex perturbed}, and denote by $\vartheta^\rho=\vartheta(\theta,\rho,g)$ the corresponding perturbed section. Let $\vartheta^H=\vartheta(\theta,H,g_H,\varepsilon)$ be any perturbed section as in Theorem \ref{theorem perturbed section}. Then 
\begin{equation}\label{equation chain homotopy equiv between non-exact and exact ref pairs}
\mathrm{CN}_\bullet(\vartheta^\rho_a,\mathcal{A}) \simeq \mathrm{CN}_\bullet(\vartheta^H_a,\mathcal{A}), \quad \forall a \in \mathcal{A}.
\end{equation}
Let $\theta_i \colon \mathcal{A} \to \Omega^1(M)$ be any other two sections $i=0,1$ with reference pairs $(\rho_i,g_i)$ satisfying the conditions of Corollary \ref{corollary other chain complex perturbed}. Then for any two choices $\vartheta^{\rho_i}=\vartheta(\theta_i,\rho_i,g_i)$ one has
\begin{equation}\label{equation chain homotopy equiv between non-exact and non-exact ref pairs}
\mathrm{CN}_\bullet(\vartheta^{\rho_0}_a,\mathcal{A}) \simeq \mathrm{CN}_\bullet(\vartheta^{\rho_1}_a,\mathcal{A}), \quad \forall a \in \mathcal{A}.
\end{equation}
Moreover, both \eqref{equation chain homotopy equiv between non-exact and exact ref pairs} and \eqref{equation chain homotopy equiv between non-exact and non-exact ref pairs} continue to hold in the restricted case $\mathcal{B} \subseteq \mathcal{A}$.
\end{theorem}

\begin{proof}
The proof idea is again arguing via continuations as in Theorem \ref{theorem independence of chain complexes} above -- we will use the latter as carbon copy and adapt the same notation. Define $$\vartheta^s=(1-h(s)) \cdot \vartheta^\rho+h(s) \cdot \vartheta^H.$$ By the same logic as in Theorem \ref{theorem independence of chain complexes}, it suffices to control the expression 
\begin{equation}\label{equation TERM TO CONTROL}
\int_{\gamma} \vartheta^s_b-\vartheta^s_a,
\end{equation}
for all $b \in \mathcal{A}$, in order to get the desired continuation chain map to conclude \eqref{equation chain homotopy equiv between non-exact and exact ref pairs}. For this purpose we define 
\begin{align*}
\mathcal{S}^\rho &:=\ \left\lbrace s \in (-\infty,0) \, \big | \,\Vert \nabla^{g_0} \vartheta^{\rho}_a (\gamma(s)) \Vert_{g_0} \geq \frac{C_\rho}{8} \right\rbrace, \\ \mathcal{S}^H&:=\ \left\lbrace s \in (1,+\infty) \, \big | \,\Vert \nabla^{g_1} \vartheta^H a (\gamma(s)) \Vert_{g_1} \geq \frac{C_H}{\varepsilon \cdot 8} \right\rbrace.
\end{align*}
Here $g_0$ and $g_1$ (abusively) denote Riemannian metrics $g_{\vartheta^\rho_a}$ and $g_{\vartheta^H_a}$ coming from Proposition \ref{proposition perturbation for the whole section of polytope}.

Observe that by assumption and choice of $(\vartheta^\rho_a,g_0)$ we have
$$\Vert \vartheta^{\rho}_b- \rho \Vert_{g_0} \leq 5 \cdot D_0 \cdot \Vert \theta_b- \rho \Vert_g \leq \frac{5 \cdot D_0 \cdot C_\rho}{D_0 \cdot 1000},$$ see proof of Corollary \ref{corollary other chain complex perturbed} and Proposition \ref{proposition perturbation for the whole section of polytope}.

Just as in the proof of Theorem \ref{theorem independence of chain complexes} there is no contribution to \eqref{equation TERM TO CONTROL} for $s$ in the complement of $[0,1] \cup \mathcal{S}^{\rho} \cup \mathcal{S}^H$. At the same time, we can again bound 

$$\bigg \vert \int_{\mathcal{S}^\alpha} \left(\vartheta^s_b-\vartheta^s_a \right) \, \dot{\gamma}(s) \, ds \bigg \vert < \frac{1}{10} \cdot E(\gamma), \, \bigg \vert \int_{\mathcal{S}^\alpha} \left(\vartheta^s_b-\vartheta^s_a \right) \, \dot{\gamma}(s) \, ds \bigg \vert < \frac{1}{10} \cdot E(\gamma),$$
and
$$\bigg \vert \int_{[0,1]} \left(\vartheta^s_b-\vartheta^s_a \right) \, \dot{\gamma}(s) \, ds \bigg \vert \leq F \cdot E(\gamma)^{\frac{1}{2}}.$$ This suffices to obtain the desired control over \eqref{equation TERM TO CONTROL} and proves \eqref{equation chain homotopy equiv between non-exact and exact ref pairs}, see proof of Theorem \ref{theorem independence of chain complexes} for more details. Last but not least, \eqref{equation chain homotopy equiv between non-exact and non-exact ref pairs} follows by applying \eqref{equation chain homotopy equiv between non-exact and exact ref pairs} twice:
$$\mathrm{CN}_\bullet(\vartheta^{\rho_0}_a,\mathcal{A}) \simeq \mathrm{CN}_\bullet(\vartheta^H_a,\mathcal{A}) \simeq \mathrm{CN}_\bullet(\vartheta^{\rho_1}_a,\mathcal{A}), \quad \forall a \in \mathcal{A}.$$

\end{proof}

Theorem \ref{theorem independence of chain complexes} and Theorem \ref{theorem indep of perturb with exact vs non-exact ref. pair} readily imply:

\begin{corollary}\label{corollary independence of data}
Let $\theta_i \colon \mathcal{A} \to \Omega^1(M)$ with $i=0,1$ be two sections and $\vartheta^i$ associated perturbations as in Theorem \ref{theorem perturbed section} (or Corollary \ref{corollary other chain complex perturbed}). Then the resulting polytope Novikov homologies are isomorphic:
$$\mathrm{HN}_\bullet(\vartheta^0_a,\mathcal{A}) \cong \mathrm{HN}_\bullet(\vartheta^1_a,\mathcal{A}), \quad \forall a \in \mathcal{A}.$$
\end{corollary}

Corollary \ref{corollary independence of data} is the analogue to the independence of Morse-Smale pairs $(\alpha,g)$ in the case of ordinary Novikov homology. The latter can also be recovered from the former by taking the trivial polytope $\mathcal{A}=\langle a \rangle$. Nevertheless, keeping track of the section $\theta$, or rather its perturbations, will prove useful, especially when establishing the commutative diagram in the Main Theorem \ref{theorem main technical theorem}.

\subsection{Non-exact deformations and proof of the Main Theorem}\label{subsection Non-exact deformatnios}

The power of the polytope machinery will become evident in this subsection -- roughly speaking, the notion of polytopes allows us to compare the Novikov homologies coming from two different cohomology classes, see Main Theorem \ref{theorem main technical theorem}. In Section \ref{section Applications of the Main Theorem} we present some applications of the Main Theorem \ref{theorem main technical theorem}.

\begin{theorem}[Main Theorem]\label{theorem main technical theorem}
Let $\mathcal{A}\subset H^1_{\mathrm{dR}}(M)$ be a polytope and $\theta \colon \mathcal{A} \to \Omega^1(M)$ a section.
Then there exists a perturbation $\vartheta \colon \mathcal{A} \to \Omega^1(M)$ of $\theta$ such that for every subpolytope $\mathcal{B} \subseteq \mathcal{A}$ and any two cohomology classes $a, \, b \in \mathcal{A}$ there exists a commutative diagram

\begin{center}
\begin{tikzcd}
\mathrm{CN}_\bullet(\vartheta_a,\mathcal{A}) \arrow[d, hook, "\iota_\mathcal{B}"] \arrow[r, "\simeq"] & \mathrm{CN}_\bullet(\vartheta_b,\mathcal{A}) \arrow[d, hook, "\iota_\mathcal{B}"] \\
\mathrm{CN}_\bullet \left(\vartheta_a,\mathcal{A}  |_{\mathcal{B}} \right) \arrow[r,"\simeq"] & \mathrm{CN}_\bullet \left( \vartheta_b,\mathcal{A}  |_{\mathcal{B}} \right).
\end{tikzcd}
\end{center}
where the horizontal maps are Novikov-linear chain homotopy equivalences. In particular
\begin{center}
\begin{tikzcd}
\mathrm{HN}_\bullet(\vartheta_a,\mathcal{A}) \arrow[d, hook] \arrow[r, "\cong"] & \mathrm{HN}_\bullet(\vartheta_b,\mathcal{A}) \arrow[d, hook] \\
\mathrm{HN}_\bullet \left(\vartheta_a,\mathcal{A}  |_{\mathcal{B}} \right) \arrow[r,"\cong"] & \mathrm{HN}_\bullet \left( \vartheta_b,\mathcal{A}  |_{\mathcal{B}} \right).
\end{tikzcd}
\end{center}
with all the maps being Novikov-linear.
\end{theorem}

\begin{proof}
Most of the ideas have already been established in the previous subsection, especially in the proof of Theorem \ref{theorem perturbed section} and Theorem \ref{theorem independence of chain complexes}. Fix a reference Morse-Smale pair $(H,g)$ and pick a perturbation $\vartheta=\vartheta(\theta,H,g,\varepsilon)$ as in Theorem \ref{theorem perturbed section}. Let $h \colon [0,1] \to \R$ be a smooth function as in \eqref{equation smooth monotone function h} and define a homotopy
$$\vartheta^s_{ab}:=(1-h(s)) \cdot \vartheta_a+h(s) \cdot \vartheta_b, \quad \forall s \in \R.$$ Pick $g_s$ a smooth homotopy connecting the two metrics $g_{\vartheta_a}$ and $g_{\vartheta_b}$ and assume that $(\vartheta^s_{ab},g_s)$ is regular (see Remark \ref{remark homotopy regular no problemo}). The idea now is to show that the chain continuation map $$\Psi^{ba} \colon \mathrm{CN}_\bullet(\vartheta_a,g_{\vartheta_a},\mathcal{A}) \longrightarrow \mathrm{CN}_\bullet(\vartheta_b,g_{\vartheta_b},\mathcal{A})$$ associated to the regular homotopy $(\vartheta^s_{ab},g_s)$ is well defined and makes the desired diagram commute. The argument that $\Psi^{ba}$ is a well defined Novikov chain map is the same as in Theorem \ref{theorem independence of chain complexes} and follows by controlling terms of the form 
$$\int_{\gamma} \vartheta_c-\vartheta^s_{ab}, \text{ with } c \in \mathcal{A}.$$
Note that this time around we do not need to put an $s$-dependence on $\vartheta_c$ (this corresponds to $\vartheta_b$ in the proof of Theorem \ref{theorem independence of chain complexes}), since the endpoints of $\vartheta^s_{ab}$ have the same zeros as $\vartheta_c$, namely $Z(\vartheta_c)=\mathrm{Crit}(H)$. The $s$-dependence on $\vartheta_c$ is the only bit that used the cohomologous assumption in Theorem \ref{theorem independence of chain complexes}, and indeed, the remaining part of the proof is verbatim the same and is thus omitted.

We use the very same homotopy to define a chain continuation $$\Psi^{ba} |_{\mathcal{B}} \colon \mathrm{CN}_\bullet \left( \vartheta_a,\mathcal{A}  |_{\mathcal{B}}\right) \longrightarrow \mathrm{CN}_\bullet \left(\vartheta_b,\mathcal{A}  |_{\mathcal{B}} \right).$$ From this we obtain the following commutative diagram on the chain level:

\begin{center}
\begin{tikzcd}
\mathrm{CN}_\bullet(\vartheta_a,\mathcal{A}) \arrow[d, hook, "\iota_{\mathcal{B}}"] \arrow[r, "\Psi^{ba}"] & \mathrm{CN}_\bullet(\vartheta_b,\mathcal{A}) \arrow[d, hook, "\iota_{\mathcal{B}}"] \\
\mathrm{CN}_\bullet \left(\vartheta_a,\mathcal{A}  |_{\mathcal{B}} \right) \arrow[r,"\Psi^{ba} |_{\mathcal{B}}"] & \mathrm{CN}_\bullet \left( \vartheta_b,\mathcal{A}  |_{\mathcal{B}} \right).
\end{tikzcd}
\end{center}
Here the $\iota_\mathcal{B}$ denote the inclusions \eqref{equation chain inclusions of restrictions}, which are chain maps (cf. Corollary \ref{corollary restricted polytopes also give well defined chain complexes}). By symmetry and the standard argument, we get continuations $\Psi^{ab}$ and $\Psi^{ab} |_{\mathcal{B}}$ in the opposite direction by reversing the underlying regular homotopy. It is also easy to see that continuation maps are linear over the underlying Novikov ring. This proves that the two horizontal chain maps above define the desired chain homotopy equivalences. In particular, the chain diagram above induces the desired diagram in homology and thus concludes the proof.

\end{proof}

\begin{remark}
In light of Theorem \ref{theorem independence of chain complexes}, Theorem \ref{theorem indep of perturb with exact vs non-exact ref. pair} and Corollary \ref{corollary independence of data} one can upgrade Theorem \ref{theorem main technical theorem} and use different sections on the left and on the right of the the diagram, i.e.
\begin{center}
\begin{tikzcd}
\mathrm{CN}_\bullet(\vartheta^0_a,\mathcal{A}) \arrow[d, hook, "\iota_{\mathcal{B}}"] \arrow[r, "\simeq"] & \mathrm{CN}_\bullet(\vartheta^1_b,\mathcal{A}) \arrow[d, hook, "\iota_{\mathcal{B}}"] \\
\mathrm{CN}_\bullet \left(\vartheta^0_a,\mathcal{A}  |_{\mathcal{B}} \right) \arrow[r,"\simeq"] & \mathrm{CN}_\bullet \left( \vartheta^1_b,\mathcal{A}  |_{\mathcal{B}} \right).
\end{tikzcd}
\end{center}
The upper and lower chain homotopy equivalences however, do come from compositions of chain continuations rather than genuine chain continuations.
\end{remark}

\subsection{Twisted Novikov complex}\label{subsection The twisted Novikov complex}

Throughout this subsection we shall assume that $\theta \colon \mathcal{A} \to \Omega^1(M)$ has already been perturbed as in Theorem \ref{theorem perturbed section}.\footnote{Strictly speaking, we could also use perturbations coming from Corollary \ref{corollary other chain complex perturbed} at the expanse of working with small sections. For the sake of exposition we refrain from stating this explicitily.} We present an alternative description of $\mathrm{CN}_\bullet(\theta_a,\mathcal{A})$ by means of local coefficients. For an extensive treatment of local coefficients we recommend \cite{Whitehead1978} and \cite[Chapter 2]{banyaga2019} in the case of Morse homology.

\begin{definition}\label{definition Novikov ring via formal variable t}
Let $G \subseteq \R$ be an additive subgroup. Then define 
\begin{equation*}
\mathrm{Nov}(G;\Z):=\mathrm{Nov}(G)
\end{equation*}
as the ring\footnote{The addition and multiplication of $\mathrm{Nov}(G)$ are the obvious ones: $n_g \, t^g+m_g \,t^g=(n_g+m_g) \, t^g$ and $(n_g \,t^g) \cdot (n_h \, t^h):=n_g \cdot n_h \, t^{g+h}$.} consisting of formal sums $\sum_{g \in G} n_g t^g$ with $n_g \in \Z$, satisfying the finiteness condition
\begin{equation*}
\forall c \in \R  \colon \; \left\{ g \, \big | \, n_g \neq 0, \, g< c \right\} \text{ is finite.}
\end{equation*}
\end{definition}
Whenever $G$ is the image of a period homomorphism $\Phi_a \colon \pi_1(M) \to \R$ we write 
\begin{equation*}
\mathrm{Nov}(a):=\mathrm{Nov}(G), \quad G=\mathrm{im}(\Phi_a).
\end{equation*} 
It turns out that $\mathrm{Nov}(a)$ is isomorphic to $\Lambda_a$, where the isomorphism is given by sending a deck transformation $A \in \Gamma_a$ to $t^{\Phi_a(A)}$ -- both finiteness conditions match and we obtain:
\begin{proposition}\label{proposition novikov rings match}
For any cohomology class $a \in H^1_{\mathrm{dR}}(M)$ we have
\begin{equation*}
\Lambda_a \cong \mathrm{Nov}(a),
\end{equation*}
as rings.
\end{proposition}

In view of Proposition \ref{proposition novikov rings match} we will also refer to $\mathrm{Nov}(a)$ as Novikov ring of $a$. Inspired by the definition of $\mathrm{Nov}(a)$ we will now define yet another ring $\mathrm{Nov}(\mathcal{A})$, which will be isomorphic to $\Lambda_\mathcal{A}$ almost by definition.

\begin{definition}
Let $\mathcal{A}:=\langle a_0, \dots, a_k \rangle$ be a polytope. Define
\begin{equation*}
\mathrm{Nov}(\mathcal{A};\Z):=\mathrm{Nov}(\mathcal{A}),
\end{equation*}
as the ring consisting of elements 
\begin{equation}\label{equation formal sums new novikov ring}
\sum_{A \in \Gamma_\mathcal{A}} n_A \, t_0^{\Phi_{a_0}(A)} \cdots \, t_k^{\Phi_{a_k}(A)}, \quad n_A \in \Z,
\end{equation}
with a multi finiteness condition
\begin{equation*}
\forall l=0,\dots,k, \, \forall c \in \R  \colon \; \left\{ A \, \big | \, n_A \neq 0, \, \Phi_{a_l}(A)< c \right\} \text{ is finite.}
\end{equation*}
\end{definition}

Similarly, for any subpolytope $\mathcal{B}\subseteq \mathcal{A}$ we define a (potentially) larger group
\begin{equation*}
\mathrm{Nov}\left(\mathcal{A}  |_{\mathcal{B}}\right) \supseteq \mathrm{Nov}(\mathcal{A})
\end{equation*}
 consisting of the same formal sums \eqref{equation formal sums new novikov ring}, but with a (potentially) less restrictive multi finiteness condition:
\begin{equation*}
\text{For all } l \text{ such that } a_l \in \mathcal{B} \text{ and } c \in \R \colon \; \left\{A \, \big | \, n_A \neq 0, \, \Phi_{a_l}(A) < c \right\} \text{ is finite.}
\end{equation*}
Analogously to Proposition \ref{proposition novikov rings match} we have:\footnote{Note that $A=B$ in $\Gamma_{\mathcal{A}}$ if and only if $\Phi_{a_l}(A)=\Phi_{a_l}(B)$ for all $l=0,\dots,k$.}

\begin{proposition}\label{proposition novikov ring GENERAL agree}
For any polytope $\mathcal{A}$ we have
\begin{equation*}
\Lambda_\mathcal{A} \cong \mathrm{Nov}(\mathcal{A})
\end{equation*}
as rings. Similarly we obtain
\begin{equation*}
\Lambda_{\mathcal{A} |_{\mathcal{B}}} \cong \mathrm{Nov}(\mathcal{A} |_{\mathcal{B}})
\end{equation*}
as well.
\end{proposition}

Let us mention that these Novikov rings are \emph{commutative} rings since $\Gamma_{\mathcal{A}}$ is abelian. Indeed, the commutator subgroup of $\pi_1(M)$ is contained in every kernel $\ker (a_l)$, in particular $$ABA^{-1}B^{-1} \in \bigcap_{l=0}^k \ker (a_l), \text{ thus  }ABA^{-1}B^{-1}=0 \in \Gamma_{\mathcal{A}}.$$

Each polytope $\mathcal{A}$ comes with a representation
\begin{equation*}
\rho_\mathcal{A} \colon \pi_1(M,x_0) \times \mathrm{Nov}(\mathcal{A}) \to \mathrm{Nov}(\mathcal{A}), \quad \rho_{\mathcal{A}}(\eta,\lambda)=t_0^{-\Phi_{a_0}(\eta)} \cdots \, t_k^{-\Phi_{a_k}(\eta)} \cdot \lambda.
\end{equation*}
Sometimes we will write out $\Phi_{a_l}(\eta)=\int_{\eta} a_l$. The importance of the minus sign will become clear in Definition \ref{definition twisted complexes}.  To any such representation one can associate a \emph{local coefficient system} 
\begin{equation*}
\underline{\mathrm{Nov}}(\mathcal{A}) \colon \Pi_1(M) \to \mathsf{mod}_{\mathrm{Nov}(\mathcal{A})},
\end{equation*}
which is unique up to isomorphisms of local coefficients. We briefly recall the construction:\footnote{This stems from the construction of a category equivalence between the fundamental groupoid and the fundamental group of a sufficiently nice topological space. This procedure allows to switch back and forth between local coefficients and representations, see \cite[Page 17]{May1999}, \cite[Chapter 2]{banyaga2019} and \cite{eilenberg1947} for more details.} fix a basepoint $x_0 \in M$ and pick for every $x \in M$ a homotopy class of paths $\{\eta_x\}$ relative to the endpoints $x_0$ and $x$. Set
\begin{equation*}
\underline{\mathrm{Nov}}(\mathcal{A})(x):=\mathrm{Nov}(\mathcal{A}), \quad \forall x \in M.
\end{equation*} For every homotopy class $\{\gamma\}$ relative to the endpoints $x$ and $y$, we first define a loop
\begin{equation*}
\underline{\gamma}:=\eta_x*\gamma*\eta_y^{-1} \colon S^1 \to M,
\end{equation*}
based at $x_0$, and then define
\begin{equation}\label{equation novikov local coeff twist}
\underline{\mathrm{Nov}}(\mathcal{A})(\gamma) \colon \underline{\mathrm{Nov}}(\mathcal{A})(y) \longrightarrow \underline{\mathrm{Nov}}(\mathcal{A})(x), \quad \underline{\mathrm{Nov}}(\mathcal{A})(\gamma)(\, \cdot \,) :=\rho_{\mathcal{A}} \left(\underline{\gamma}, \, \cdot \, \right).\footnote{Note that $\Gamma(\gamma * \eta)=\Gamma(\gamma) \circ \Gamma(\eta)$.}
\end{equation}
Taking a closer look at \eqref{equation novikov local coeff twist} reveals that the Novikov ring  isomomorphism $\underline{\mathrm{Nov}}(\mathcal{A})(\gamma)$ is given by multiplication with 
$$t_0^{-\int_{\underline{\gamma}} a_0} \cdots \, t_k^{-\int_{\underline{\gamma}}a_k} \in \Z[\Gamma_{\mathcal{A}}] \subseteq  \mathrm{Nov}(\mathcal{A}),$$ hence we may also view it as a  $\Z[\Gamma_{\mathcal{A}}]$-module isomorphism. 

\begin{remark}\label{remark better local coeff}
Usually local coefficients are considered to take values in the category of abelian groups and are often called ``bundle of abelian groups". Mapping into $\mathsf{mod}_{\mathrm{Nov}(\mathcal{A})}$ will allow us to obtain actual Novikov-module isomorphisms at times where working with bundle of abelian groups would merely grant group isomorphisms. 
\end{remark}

With this we define the anticipated twisted Novikov complexes.

\begin{definition}\label{definition twisted complexes}
Let $\theta \colon \mathcal{A} \to \Omega^1(M)$ a section as above. We define the \textbf{twisted Novikov chain complex groups} by
\begin{equation*}
\mathrm{CN}_\bullet\left(\theta_a,\underline{\mathrm{Nov}}(\mathcal{A}) \right):=\bigoplus_{x \in Z(\theta_a)} \mathrm{Nov}(\mathcal{A}) \, \langle x \rangle, \quad \forall a \in \mathcal{A}.
\end{equation*}
The \textbf{twisted boundary oparator} $\underline{\partial}=\underline{\partial}_{\theta_a}$ is defined by
\begin{equation*}
\underline{\partial}(\lambda \, x):=\sum_{y, \, \gamma \in \underline{\mathcal{M}}(x,y;\theta_a)} \underline{\mathrm{Nov}}(\mathcal{A})(\gamma^{-1})(\lambda) \, y=\sum_{y, \, \gamma \in \underline{\mathcal{M}}(x,y;\theta_a)} t_0^{\int_{\underline{\gamma}} a_0} \, \cdots \; t_k^{\int_{\underline{\gamma}} a_k} \cdot \lambda  \, y, \quad \forall \lambda \in \mathrm{Nov}(\mathcal{A}).
\end{equation*}
The twisted chain complexes
\begin{equation*}
\left(\mathrm{CN}_\bullet\left(\theta_a,\underline{\mathrm{Nov}}\left(\mathcal{A}  |_{\mathcal{B}}\right) \right), \underline{\partial}_{\theta_a}  \right), \quad \forall a \in \mathcal{A}
\end{equation*}
are defined analogously. The corresponding \textbf{twisted Novikov homologies} are denoted by 
\begin{equation*}
\mathrm{HN}_\bullet\left(\theta_a,\underline{\mathrm{Nov}}(\mathcal{A}) \right) \text{ and  } \mathrm{HN}_\bullet\left(\theta_a,\underline{\mathrm{Nov}}\left(\mathcal{A}  |_{\mathcal{B}}\right) \right), \quad \forall a \in \mathcal{A}, \, \mathcal{B} \subseteq \mathcal{A}.
\end{equation*}
For every flow line $\gamma \in \underline{\mathcal{M}}(x,y;\theta_a)$ we call $$\underline{\mathrm{Nov}}(\mathcal{A})(\gamma^{-1})=t_0^{\int_{\underline{\gamma}} a_0}\cdots t_k^{\int_{\underline{\gamma}} a_k}$$ the \textbf{Novikov twist} of $\gamma$.
\end{definition}

\begin{remark}
A priori it is not clear that $\underline{\partial}$ maps into the prescribed chain complex. We will prove this in the next subsection, see Proposition \ref{proposition twisted and polytop chain complexes agree}.
\end{remark}

The Novikov twist of $\gamma$ determines the lifting behaviour of $\gamma$. Indeed, if $\gamma_0, \gamma_1$ are two paths from $x$ to $y$ with unique lifts $\tilde{\gamma}_1(0)=\tilde{\gamma}_0(0)=\tilde{x}$, then:

\begin{align*}
\tilde{\gamma}_0(1)=\tilde{\gamma}_1(1) &\iff \int_{\gamma_0*\gamma_1^{-1}} a_l=0, \, \forall l=0,\dots,k  \\
&\iff \int_{\gamma_0}a_l =\int_{\gamma_1}a_l, \, \forall l=0,\dots,k \\
&\iff \int_{\underline{\gamma_0}}a_l =\int_{\underline{\gamma_1}}a_l, \, \forall l=0,\dots,k \\
&\iff t_0^{\int_{\underline{\gamma_0}} a_0} \cdots \, t_k^{\int_{\underline{\gamma_0}}a_k}=t_0^{\int_{\underline{\gamma_1}} a_0} \cdots \, t_k^{\int_{\underline{\gamma_1}}a_k}.
\end{align*}
This will be key in the next subsection.

As the next examples shows, we recover the \emph{twisted Morse chain complex} and its homology as a special case.

\begin{example}\label{example twisted morse as special case}
Pick $\mathcal{A}=\langle 0, a \rangle$ and a section $\theta \colon \mathcal{A} \to \Omega^1(M)$. Then 
\begin{equation*}
\mathrm{Nov}(\mathcal{A}|_a) \cong \mathrm{Nov}(a) \cong \Lambda_a,
\end{equation*}
thanks to Proposition \ref{proposition novikov rings match} and  Example \ref{example polytope consisting of point}. Furthermore, the twisted chain complex $$\mathrm{CN}_\bullet \left(\theta_0,\underline{\mathrm{Nov}}(\mathcal{A}|_a)\right)=\mathrm{CN}_\bullet \left(\theta_0,\underline{\mathrm{Nov}}(a)\right)$$ agrees with the twisted Morse complex $$\mathrm{CM}_\bullet\left(h, \underline{\mathrm{Nov}}(a) \right), \text{ where } \theta_0=dh.$$
\end{example}

\subsection{Comparing twisted and polytope complexes}

In general, the same issues as in subsection \ref{subsection Novikov homology with polytopes} arises when trying to prove that the twisted Novikov complexes are well defined: it is not clear whether $\underline{\partial}$ maps into the desired chain complex. This is a non-issue in the special case of $\theta_0=dh$, i.e. twisted Morse homology -- the reason is that the $0$-dimensional moduli spaces $\underline{\mathcal{M}}(x,y;h)$ are compact, hence finite. Compare this to \cite{ritter2009}. In the following however, we will see that the twisted chain \emph{groups} can always be identified with $\mathrm{CN}_\bullet(\theta_a,\mathcal{A})$ so that $\underline{\partial}$ and $\partial$ agree, which then resolves the well definedness issue by Theorem \ref{theorem perturbed section}. In other words, the twisted chain complex is an equivalent description of the polytope chain complex.

\begin{proposition}\label{proposition twisted and polytop chain complexes agree}
Let $\theta \colon \mathcal{A} \to \Omega^1(M)$ be a section as above. Then the twisted and polytope Novikov chain groups are isomorphic
\begin{equation}\label{equation correspondende twisted and polytope chain complexes}
\mathrm{CN}_\bullet\left(\theta_a,\underline{\mathrm{Nov}}(\mathcal{A}) \right) \longleftrightarrow \mathrm{CN}_\bullet(\theta_a,\mathcal{A}),
\end{equation}
as Novikov-modules. The isomorphism is preserved upon restrictions $\mathcal{A}  |_{\mathcal{B}}$, and $\partial= \underline{\partial}$ up to the identification \eqref{equation correspondende twisted and polytope chain complexes}.

In particular, the twisted Novikov chain complexes are well defined with
\begin{equation*}
\mathrm{HN}_\bullet\left(\theta_a,\underline{\mathrm{Nov}}(\mathcal{A}) \right)=\mathrm{HN}_\bullet(\theta_a,\mathcal{A}) \text{ and } \mathrm{HN}_\bullet \left(\theta_a,\underline{\mathrm{Nov}}\left(\mathcal{A}  |_{\mathcal{B}}\right) \right)=\mathrm{HN}_\bullet\left(\theta_a,\mathcal{A}  |_{\mathcal{B}}\right), 
\end{equation*}
for all $a \in \mathcal{A}$ and subpolytopes $\mathcal{B} \subseteq \mathcal{A}$.
\end{proposition}

\begin{proof}
First of all recall that we can view the $i$-th polytope Novikov chain groups as finitely-generated Novikov-modules by fixing a finite set of preferred lifts $\tilde{x}_m \in \pi^{-1}(x_m)$, for each zero $x_m$ of $\theta_a$ of index $i$:
\begin{equation*}
\mathrm{CN}_i(\theta_a,\mathcal{A}) \cong \bigoplus_m \Lambda_\mathcal{A} \langle \tilde{x}_m \rangle, \quad \text{ as Novikov ring modules}.
\end{equation*}
Since $\Lambda_\mathcal{A} \cong \mathrm{Nov}(\mathcal{A})$ (cf. Proposition \ref{proposition novikov ring GENERAL agree}), we end up with
$$\mathrm{CN}_i(\theta_a,\mathcal{A}) \cong \bigoplus_m \Lambda_\mathcal{A} \langle \tilde{x}_m \rangle \cong \bigoplus_m \mathrm{Nov}(\mathcal{A}) \langle x_m \rangle =\mathrm{CN}_i\left(\theta_a,\underline{\mathrm{Nov}}(\mathcal{A})\right). $$
Both boundary operators $\partial$ and $\underline{\partial}$ are $\Lambda_\mathcal{A}$- and $\mathrm{Nov}(\mathcal{A})$-linear, thus it suffices to compare $\partial \tilde{x}_m$ and $\underline{\partial} x_m$. On one hand we have
\begin{equation}\label{equation boundary polytope compare}
\partial \tilde{x}_m=\sum_n \lambda_{m,n} \tilde{y}_n, \text{ with } \lambda_{m,n}=\sum_{A \in \Gamma_\mathcal{A}}\#_{\mathrm{alg}} \, \underline{\mathcal{M}}\left(\tilde{x}_m,A \tilde{y}_n;\tilde{f}_{\theta_a}\right) \, A \in \Lambda_{\mathcal{A}},
\end{equation}
and the other hand
\begin{equation}\label{equation boundary twisted compare}
\underline{\partial} x_m=\sum_n \left( \sum_{\gamma \in \underline{\mathcal{M}}\left(x_m,y_n,\theta_a \right)} t_0^{\int_{\underline{\gamma}} a_0 } \, \cdots \; t_k^{\int_{\underline{\gamma}} a_k}  \right) \, y_n.
\end{equation}
In the previous subsection we have seen that the Novikov twist $t_0^{\int_{\underline{\gamma}} a_0} \, \cdots \; t_k^{\int_{\underline{\gamma}} a_k}$ of $\gamma$ uniquely determines the lifting behaviour of $\gamma$. Thus if $\tilde{\gamma}$ denotes the unique lift which starts at $\tilde{x}_n$ and ends at some $\tilde{y}$, we get
$$\tilde{y}=A\tilde{y}_n,$$ with $A \in \Gamma_\mathcal{A} \subset \Lambda_\mathcal{A}$ corresponding to $t_0^{\int_{\underline{\gamma}} a_0} \, \cdots \; t_k^{\int_{\underline{\gamma}}a_k} \in \mathrm{Nov}(\mathcal{A})$. This proves that \eqref{equation boundary polytope compare} and \eqref{equation boundary twisted compare} agree up to identifying the respective isomorphic Novikov rings. The same proof also shows that the restricted complexes associated to $\mathcal{A} |_{\mathcal{B}}$ agree.

With the identification of twisted and polytope complexes at hand, we can invoke Theorem \ref{theorem perturbed section} (recall that we already assumed that $\theta$ is perturbed accordingly) and deduce that the twisted chain complex is well defined. By the first part it follows that the corresponding homologies agree. This finishes the proof.

\end{proof}

\begin{remark}\label{remark novikov weights to prove partial squared equal to 0}
The twisted complex can be used to deduce properties of the polytope complex and vice versa.  For instance, trying to prove that $\partial^2=0$ is equivalent to proving $\underline{\partial}^2=0$, which has a far more pleasant proof -- the reason is that the Novikov twists are nicely behaved with respect to the compactification of the moduli spaces, see for instance \cite[Proposition 1]{ritter2009}.
\end{remark}

\section{Applications of the Main Theorem}\label{section Applications of the Main Theorem}

\subsection{The 0-vertex trick and the Morse-Eilenberg Theorem}\label{subsection The 0-vertex trick and the Morse-Eilenberg Theorem}

All the applications we are about to present boil down to what we call the \emph{0-vertex trick}. The idea is to relate the polytope chain groups to twisted Morse chain groups  by extending the underlying polytope $\mathcal{A}$ with $0 \in H^1_{\mathrm{dR}}(M)$ as an additional vertex, and then using the Main Theorem \ref{theorem main technical theorem}. This trick suffices to prove the (twisted) Novikov Morse Homology Theorem (cf. Theorem \ref{theorem NMH Theorem} and Corollary \ref{corollary NHT for real}).

For the polytope Novikov Principle (cf. Theorem \ref{theorem Polytop Novikov Principle}) we shall need a Morse variant of the Eilenberg Theorem \cite[Theorem 24.1]{eilenberg1947}, which has been proven in \cite[Theorem 2.21]{banyaga2019}. We will state and prove a slightly stronger version in the context of Novikov theory down below, see Lemma \ref{lemma Morse-Eilenberg}.

\begin{lemma}[0-vertex trick]\label{lemma 0 vertex trick}
Let $\theta \colon \mathcal{A} \to \Omega^1(M)$ be a section, $\mathcal{B} \subseteq \mathcal{A}$ a subpolytope, and $\mathcal{A}^0$ the polytope spanned by the vertices of $\mathcal{A}$ and $0$.\footnote{Note that if $0$ is already contained in $\mathcal{A}$, then $\mathcal{A}^0=\mathcal{A}$.} Let $\theta^0 \colon \mathcal{A}^0 \to \Omega^1(M)$ be a section extending $\theta$. Then
$$\Z[\Gamma_{\mathcal{A}}]=\Z[\Gamma_{\mathcal{A}^0}] \text{ and } \underline{\mathrm{Nov}}(\mathcal{A} |_{\mathcal{B}}) = \underline{\mathrm{Nov}}(\mathcal{A}^0 |_{\mathcal{B}}).$$
In particular, there exists a perturbed section $\vartheta^0 \colon \mathcal{A}^0 \to \Omega^1(M)$ such that
$$\mathrm{CN}_\bullet\left(\vartheta^0_a,\mathcal{A} |_{\mathcal{B}}\right)=\mathrm{CN}_\bullet\left(\vartheta^0_a,\mathcal{A}^0 |_{\mathcal{B}}\right) \simeq \mathrm{CN}_\bullet \left(\vartheta^0_0,\mathcal{A}^0 |_{\mathcal{B}} \right)=\mathrm{CM}_\bullet\left(h,\underline{\mathrm{Nov}}(\mathcal{A} |_{\mathcal{B}}) \right), \quad \forall a \in \mathcal{A},$$
as chain complexes, where $dh=\vartheta^0_0$.
\end{lemma}

\begin{proof}
Since $\ker(0)=\pi_1(M)$, adding $0$ as vertex does not affect the underlying abelian cover, i.e. $\widetilde{M}_\mathcal{A}=\widetilde{M}_{\mathcal{A}^0}$, also see Example \ref{example polytope consisting of point}. Thus the deck transformation groups $\Gamma_{\mathcal{A}}$ and $\Gamma_{\mathcal{A}^0}$ are equal and so are the respective group rings. The finiteness conditions for both $\mathrm{Nov}(\mathcal{A} |_{\mathcal{B}})$ and $\mathrm{Nov}(\mathcal{A}^0 |_{\mathcal{B}})$ is determined by the subpolytope $\mathcal{B} \subseteq \mathcal{A}$, and thus by the group ring equality we also deduce $$\mathrm{Nov}(\mathcal{A} |_{\mathcal{B}})=\mathrm{Nov}(\mathcal{A}^0 |_{\mathcal{B}}).$$
The equality as local coefficient systems then also follows by observing that the period homomorphism $\Phi_0$ is identically zero, hence $t^{\Phi_0(A)}=1$ for all $A \in \Gamma_{\mathcal{A}^0}$.
Pick $\vartheta^0$ as in Theorem \ref{theorem perturbed section} (or Corollary \ref{corollary other chain complex perturbed}). The first chain polytope equality follows from the Novikov rings being equal and the chain homotopy equivalence stems from the Main Theorem \ref{theorem main technical theorem}. The last equality follows from Example \ref{example twisted morse as special case} and the equality of local coefficient systems above.

\end{proof}

We conclude the subsection by stating and proving a chain-level Morse variant of the Eilenberg Theorem \cite[Theorem 2.21]{banyaga2019}

\begin{lemma}[Morse-Eilenberg Theorem]\label{lemma Morse-Eilenberg}
Let $h \colon M \to \R$ be Morse function, $\mathcal{A}$ a polytope and $\mathcal{B} \subseteq \mathcal{A}$ a subpolytope. Then
$$\mathrm{CM}_\bullet\left(\tilde{h}\right) \otimes_{\Z[\Gamma_{\mathcal{A}}]} \mathrm{Nov}(\mathcal{A} |_{\mathcal{B}}) \cong \mathrm{CM}_\bullet\left(h,\underline{\mathrm{Nov}}(\mathcal{A} |_{\mathcal{B}}) \right)$$
as chain complexes over $\mathrm{Nov}(\mathcal{A} |_{\mathcal{B}})$, where $\tilde{h}=h\circ \pi$.
\end{lemma}

The proof of \cite[Theorem 2.21]{banyaga2019} constructs a group chain isomorphism $\Psi$ and a careful inspection reveals that $\Psi$ defines a Novikov-module chain isomorphism when working with the according local coefficient system $\underline{\mathrm{Nov}}(\mathcal{A} |_{\mathcal{B}})$. Nevertheless, we decided to give a full proof of Lemma \ref{lemma Morse-Eilenberg} using the tools developed in the previous sections.

\begin{proof}[Proof of Lemma \ref{lemma Morse-Eilenberg}]
First of all observe that $\mathrm{CM}_\bullet\left(\tilde{h}\right)$ is a finite $\Z[\Gamma_\mathcal{A}]$-module
$$\mathrm{CM}_\bullet(\tilde{h})=\bigoplus_m \Z[\Gamma_\mathcal{A}] \, \langle \tilde{x}_m \rangle.$$
Here $\{\tilde{x}_m\}$ denotes a finite set of preferred lifts as in the proof of Proposition \ref{proposition twisted and polytop chain complexes agree}. Define
$$\Psi \colon \mathrm{CM}_\bullet\left(\tilde{h}\right) \otimes_{\Z[\Gamma_{\mathcal{A}}]} \mathrm{Nov}(\mathcal{A} |_{\mathcal{B}}) \longrightarrow \mathrm{CM}_\bullet\left(h,\underline{\mathrm{Nov}}(\mathcal{A} |_{\mathcal{B}}) \right), \; \Psi(\tilde{x}_m \otimes \lambda)=\lambda \, \langle x_m \rangle.$$ By the above observation $\Psi$ is a well defined $\Z[\Gamma_\mathcal{A}]$-linear map. It is clear that $\Psi$ is surjective. For injectivity we observe that $$\Psi(\tilde{x}_m \otimes \lambda)=\Psi(\tilde{x}_n \otimes \mu) \iff \lambda \, \langle x_m \rangle = \mu \, \langle x_n \rangle.$$ Therefore we must have $\lambda=\mu$ and $x_m=x_n$, hence $\tilde{x}_m=\tilde{x}_n$ -- recall that we are working with a preferred set of critical points in each fiber. This proves injectivity.

Next we show that $\Psi$ is $\mathrm{Nov}(\mathcal{A} |_{\mathcal{B}})$-linear.\footnote{We are implicitly identifying $\Lambda_{\mathcal{A} |_{\mathcal{B}}}=\mathrm{Nov}(\mathcal{A} |_{\mathcal{B}})$, see Proposition \ref{proposition novikov rings match}.} Pick $\lambda, \, \mu \in \mathrm{Nov}(\mathcal{A} |_{\mathcal{B}})$ and observe:
$$\Psi(\mu \cdot (\tilde{x}_m \otimes \lambda))=\Psi(\tilde{x}_m \otimes (\mu \cdot \lambda))=(\mu \cdot \lambda) \, \langle x_m \rangle=\mu \cdot (\lambda \, \langle x_m \rangle)=\mu \cdot \Psi(\tilde{x}_m \otimes \lambda).$$ Hence $\Psi$ is a Novikov-linear isomorphism.
We are only left to show $\underline{\partial} \circ \Psi=\Psi \circ (\partial^M \otimes \mathrm{id})$. Recall that $\partial^M$ is defined by counting Morse trajectories of $\tilde{h}$ on $\widetilde{M}_\mathcal{A}$. In particular, the boundary operator $\partial$ on $\mathrm{CN}_\bullet\left(dh,\mathcal{A}|_{\mathcal{B}}\right)$ agrees with $\partial^M$ on $\mathrm{Crit}(\tilde{h})$. Let us adopt the notation of the proof of Proposition \ref{proposition twisted and polytop chain complexes agree} and write $$\partial^M \tilde{x}_m=\partial \tilde{x}_m= \sum_n \lambda_{m,n} \, \tilde{y}_n, \quad \lambda_{m,n} \in \Z[\Gamma_{\mathcal{A}}]. \footnote{Sanity check: $\displaystyle\lambda_{m,n}=\sum_{\mathcal{A} \in \Gamma_{\mathcal{A}}} \#_{\mathrm{alg}} \, \underline{\mathcal{M}}(\tilde{x}_m,A \tilde{y}_n; \tilde{h}) \, A$, but $\displaystyle\bigcup_{A \in \Gamma_{\mathcal{A}}}\underline{\mathcal{M}}(\tilde{x}_m,A \tilde{y}_n;\tilde{h}) \cong \underline{\mathcal{M}}(x_m,y_n;h) $ and the latter is finite.}$$ Therefore 
$$\Psi \circ (\partial^M \otimes \mathrm{id}) \, \tilde{x}_m \otimes \lambda=\Psi \big( \sum_n \lambda_{m,n} \, \tilde{y}_n \otimes \lambda\big)=\Psi \big(\tilde{x}_m \otimes \sum_n \lambda_{m,n} \cdot \lambda \big)= \big( \sum_n \lambda_{m,n} \cdot \lambda \big) \, \langle x_m \rangle.$$ 
On the other hand Proposition \ref{proposition twisted and polytop chain complexes agree} says that up to identifying $\tilde{x}_m$ and $x_m$ we have $\partial \tilde{x}_m = \underline{\partial} \, x_m$, thus
$$\underline{\partial} \circ \Psi(\tilde{x}_m \otimes \lambda)=\underline{\partial}(\lambda \, \langle x_m \rangle)=\lambda \, \underline{\partial} \, x_m= \lambda \cdot \sum_n \lambda_{m,n} \, y_n= \big(\lambda \cdot \sum_n \lambda_{m,n} \big) \, \langle y_n \rangle.$$
But $\mathrm{Nov}(\mathcal{A} |_{\mathcal{B}})$ is a commutative ring, hence we conclude the chain property of $\Psi$ and thus that $\Psi$ defines a Novikov-linear chain isomorphism.

\end{proof}

\subsection{The twisted Novikov Morse Homology Theorem}\label{subsection The Novikov Morse Homology}

Using the results developed in the Section \ref{section Novikov Homology and Polytopes} and the $0$-vertex trick (cf. Lemma \ref{lemma 0 vertex trick}) we are going to prove:

\begin{theorem}\label{theorem NMH Theorem}
Let $f$ be a Morse function and $a \in H^1_{\mathrm{dR}}(M)$ a cohomolgoy class. Then for every Morse representative $\alpha \in a$ there exists a chain homotopy equivalence 
\begin{equation*}
\mathrm{CN}_\bullet (\alpha) \simeq \mathrm{CM}_\bullet \left( f, \underline{\mathrm{Nov}}(a) \right)
\end{equation*}
of Novikov-modules.
\end{theorem}

One can prove that the twisted Morse homology computes singular homology with coefficients $\underline{\mathrm{Nov}}(a)$, see \cite[Theorem 4.1]{banyaga2019}.\footnote{The authors mention in the proof of \cite[Lemma 6.30]{banyaga2019} that the isomorphism in \cite[Theorem 4.1]{banyaga2019} is an isomorphism in the category of the underlying local coefficient system, thus a Novikov-module isomorphism in our case.} Combining this with Theorem \ref{theorem NMH Theorem} and taking the homology then shows:
\begin{corollary}(Twisted Novikov Homology Theorem)\label{corollary NHT for real}
For any cohomology class $a \in H^1_{\mathrm{dR}}(M)$ there exists an isomorphism
\begin{equation*}
\mathrm{HN}_\bullet (a) \cong \mathrm{H}_\bullet(M,\underline{\mathrm{Nov}}(a))
\end{equation*}
of Novikov-modules.
\end{corollary}

\begin{remark}\label{remark proof for twisted NHT is novel}
This is a slightly different incarnation of the classical Novikov Morse Homology Theorem as Corollary \ref{corollary NHT for real} relates the Novikov homology to twisted singular homology rather than equivariant singular homology. Moreover, as the proof will show, we do \emph{not} invoke the Eilenberg Theorem (or its Morse analogue from the previous subsection) and instead produce a direct connection between the Novikov complex and the twisted Morse complex via the $0$-vertex trick -- this chain of arguments appears to be novel.
\end{remark}

\begin{proof}[Proof of Theorem \ref{theorem NMH Theorem}]
Pick $0, \,a$ in $H^1_{\mathrm{dR}}(M)$, set $\mathcal{A}:=\langle 0,a \rangle$,  and consider any section $$\theta \colon \mathcal{A} \to \Omega^1(M).$$ Up to perturbing $\theta$ we may assume that $\theta$ is a section $\vartheta^0$ (note that here $\mathcal{A}^0=\mathcal{A}$) as in Lemma \ref{lemma 0 vertex trick}. In particular, setting $\mathcal{B}=\langle a \rangle $ and invoking Lemma \ref{lemma 0 vertex trick} we get a Novikov chain homotopy equivalence
$$\mathrm{CN}_\bullet\left(\theta_a,\mathcal{A} |_a \right) \simeq \mathrm{CM}_\bullet \left(h, \underline{\mathrm{Nov}}(\mathcal{A} |_a) \right).$$
From Example \ref{example twisted morse as special case} we deduce that $\underline{\mathrm{Nov}}\left(\mathcal{A} |_a \right)=\underline{\mathrm{Nov}}(a)$, therefore we obtain
$$\mathrm{CN}_\bullet(\theta_a) \simeq \mathrm{CM}_\bullet(h,\underline{\mathrm{Nov}}(a))$$
as Novikov-modules. Twisted Morse homology, just as ordinary Morse homology, does not depend on the choice of Morse function. Indeed, using continuation methods one can prove $$\mathrm{CN}_\bullet(h,\underline{\mathrm{Nov}}(a)) \simeq \mathrm{CM}_\bullet(f,\underline{\mathrm{Nov}}(a))$$ as Novikov-modules.\footnote{See the proof of \cite[Theorem 3.9]{banyaga2019} for more details.} This proves 
$$\mathrm{CN}_\bullet(\theta_a) \simeq \mathrm{CM}_\bullet(f,\underline{\mathrm{Nov}}(a))$$ as Novikov-modules. Observing $\theta_a \in a$ and taking the homology on both sides completes the proof.

\end{proof}

\subsection{A polytope Novikov Principle}

Combining the two lemmata from Subsection \ref{subsection The 0-vertex trick and the Morse-Eilenberg Theorem} and the Main Theorem \ref{theorem main technical theorem} we will prove a \emph{polytope Novikov Principle} (see Theorem \ref{theorem Polytop Novikov Principle}). As a corollary we recover the ordinary Novikov Principle (cf. Corollary \ref{corollary ordinary NP}). In fact, the results here also cover those in the previous Subsection \ref{subsection The Novikov Morse Homology}. We opted to keep them apart to emphasise the novelty of the ``twisted" approach to the Novikov Homology Theorem (see Remark \ref{remark proof for twisted NHT is novel}).

The polytope Novikov Principle yields a new proof (in the abelian case) to a recent result due to Pajitnov \cite[Theorem 5.1]{Pajitnov2019}.

\begin{theorem}[Polytope Novikov Principle]\label{theorem Polytop Novikov Principle}
Let $\theta \colon \mathcal{A} \to \Omega^1(M)$ be any section and $\mathcal{B}\subseteq \mathcal{A}$ a subpolytope. Then there exists a perturbed section $\vartheta \colon \mathcal{A} \to \Omega^1(M)$ such that
\begin{equation}\label{equation chain homotopy Equiv}
\mathrm{CN}_\bullet\left(\vartheta_a,\mathcal{A} |_{\mathcal{B}} \right) \simeq C_\bullet\left(\widetilde{M}_\mathcal{A}\right)\otimes_{\Z[\Gamma_\mathcal{A}]}\mathrm{Nov}(\mathcal{A}|_\mathcal{B}), \quad \forall a \in \mathcal{A},
\end{equation}
as Novikov-modules.
\end{theorem}

Here $C_\bullet$ denotes the singular chain complex with $\Z$-coefficients.

\begin{proof}[Proof of Theorem \ref{theorem Polytop Novikov Principle} ]
Let $\theta^0 \colon \mathcal{A}^0 \to \Omega^1(M)$ be a section that extends $\theta$ with $\mathcal{A}^0$ the polytope generated by the vertices of $\mathcal{A}$ and $0$. By the $0$-vertex trick, i.e. Lemma \ref{lemma 0 vertex trick}, there exists a perturbed section $\vartheta^0 \colon \mathcal{A}^0 \to \Omega^1(M)$ such that
$$\mathrm{CN}_\bullet\left(\vartheta^0_a,\mathcal{A} |_{\mathcal{B}}\right) \simeq \mathrm{CM}_\bullet\left(h,\underline{\mathrm{Nov}}(\mathcal{A} |_{\mathcal{B}}) \right), \quad \forall a \in \mathcal{A}$$ as Novikov-modules. Combining this with Lemma \ref{lemma Morse-Eilenberg} we obtain
$$\mathrm{CN}_\bullet\left( \vartheta^0_a,\mathcal{A} |_{\mathcal{B}} \right) \simeq \mathrm{CM}_\bullet\left( \tilde{h} \right) \otimes_{\Z[\Gamma_\mathcal{A}]} \mathrm{Nov}(\mathcal{A} |_{\mathcal{B}}), \quad \forall a \in \mathcal{A}$$ as Novikov-modules. From standard Morse theory we know that the Morse chain complex $\mathrm{CM}_\bullet\left(\tilde{h} \right)$ is chain homotopy equivalent over $\Z[\Gamma_{\mathcal{A}}]$ to the singular chain complex $C_\bullet\left( \widetilde{M}_\mathcal{A} \right)$, see for instance \cite[Page 415]{Pajitnov2016} and \cite[Appendix]{Pajitnov1996}. Denote by
$$i \colon \mathrm{CM}_\bullet \left(\tilde{h} \right) \longrightarrow C_\bullet\left(\widetilde{M}_\mathcal{A} \right), \; j \colon C_\bullet\left(\widetilde{M}_\mathcal{A} \right) \longrightarrow \mathrm{CM}_\bullet \left( \tilde{h} \right)$$ such a chain homotopy equivalence. Then one can easily check that $i \otimes \mathrm{id}_{\mathrm{Nov}(\mathcal{A} |_\mathcal{B})}$ and $j \otimes \mathrm{id}_{\mathrm{Nov}(\mathcal{A} |_\mathcal{B})}$ define a Novikov-linear chain homotopy equivalence $$ \mathrm{CM}_\bullet \left(\tilde{h} \right) \otimes_{\Z[\Gamma_{\mathcal{A}}]} \mathrm{Nov}(\mathcal{A} |_{\mathcal{B}})  \simeq C_\bullet\left(\widetilde{M}_\mathcal{A} \right) \otimes_{\Z[\Gamma_{\mathcal{A}}]} \mathrm{Nov}(\mathcal{A} |_{\mathcal{B}}).\footnote{One can also deduce from category theory by observing that the functor $F \colon \mathsf{mod}_{\Z[\mathcal{A}]} \to \mathsf{mod}_{\mathrm{Nov}(\mathcal{A} |_{\mathcal{B}})}$, $$F(O):=O \otimes_{\Z[\Gamma_{\mathcal{A}}]} \mathrm{Nov}(\mathcal{A} |_{\mathcal{B}}),\, F(i)=i \otimes \mathrm{id}, \quad  O,P \in \mathsf{obj}\left(\mathsf{mod}_{\Z[\Gamma_{\mathcal{A}}]}\right), \, i \in \hom(O,P)$$ is additive. }$$  
Hence
$$\mathrm{CN}_\bullet\left( \vartheta^0_a,\mathcal{A} |_{\mathcal{B}} \right) \simeq \mathrm{CM}_\bullet\left( \tilde{h} \right) \otimes_{\Z[\Gamma_\mathcal{A}]} \mathrm{Nov}(\mathcal{A} |_{\mathcal{B}}) \simeq C_\bullet\left(\widetilde{M}_\mathcal{A} \right) \otimes_{\Z[\Gamma_{\mathcal{A}}]} \mathrm{Nov}(\mathcal{A} |_{\mathcal{B}}), \quad \forall a \in \mathcal{A}.$$
Setting $\vartheta:=\vartheta^0 |_{\mathcal{A}}$ finishes the proof.

\end{proof}

\begin{remark}\label{remark PNP for small sections is GOOD}
If one is interested in a particular Morse form $\omega$, then the following improvement can be made: let $(\omega,g)$ be Morse-Smale and assume that $\mathcal{A}$ is a polytope around $[\omega]$ that admits a section $\theta \colon \mathcal{A} \to \Omega^1(M)$ sufficiently close to $(\omega,g)$ in the sense of Corollary \ref{corollary other chain complex perturbed}. Denote by $\vartheta^\omega \colon \mathcal{A} \to \Omega^1(M)$ the associated perturbation with reference pair $(\omega,g)$. The $\vartheta=\vartheta^0 |_{\mathcal{A}}$ section in the proof above might come from an exact reference pair since $0$ could a priori be far away from $\theta$. However, Corollary \ref{corollary independence of data} asserts
$$\mathrm{CN}(\vartheta^{\omega}_a,\mathcal{A} |_{\mathcal{B}}) \simeq \mathrm{CN}(\vartheta_a,\mathcal{A} |_{\mathcal{B}}), \quad \forall a \in \mathcal{A}.$$ Combining this with \eqref{equation chain homotopy Equiv} for $a=[\omega]$ gives
\begin{equation*}
\mathrm{CN}_\bullet(\omega,\mathcal{A}|_{\mathcal{B}}) \simeq C_\bullet\left( \widetilde{M}_\mathcal{A}\right) \otimes_{\Z[\Gamma_{\mathcal{A}}]} \mathrm{Nov}(\mathcal{A} |_{\mathcal{B}})
\end{equation*}
since $\vartheta^\omega_{[\omega]}=\omega$.
\end{remark}

In the special case of ordinary Novikov theory, i.e. $\mathcal{A}=\langle a \rangle$ and $\mathcal{A}^0=\langle 0,a\rangle$, Theorem \ref{theorem Polytop Novikov Principle} reduces to the ordinary Novikov Principle:

\begin{corollary}(Ordinary Novikov Principle)\label{corollary ordinary NP}
Let $(\alpha,g)$ be Morse-Smale. Then
$$\mathrm{CN}_\bullet(\alpha) \simeq C_\bullet\left(\widetilde{M}_a\right) \otimes_{\Z[\Gamma_a]} \mathrm{Nov}(a).$$
\end{corollary}

\begin{proof}
Set $a=[\alpha]$, pick $\mathcal{A}=\langle a \rangle$, $\mathcal{B}=\mathcal{A}$ and $\theta \colon \mathcal{A} \to \Omega^1(M)$ the smooth section defined by $\theta_a=\alpha$. The section $\theta$ is obviously close to $(\alpha,g)$, thus Theorem \ref{theorem Polytop Novikov Principle} and Remark \ref{remark PNP for small sections is GOOD} imply $$\mathrm{CN}_\bullet(\alpha) \simeq C_\bullet\left(\widetilde{M}_a \right) \otimes_{\Z[\Gamma_a]} \mathrm{Nov}(a)$$ as Novikov-modules. 

\end{proof}
Even though it is hidden in the proof above, the main idea is still to use perturbations $\vartheta^0 \colon \mathcal{A}^0 \to \Omega^1(M)$ associated to an exact reference pair $(H,g)$. Recall that these sections $\vartheta^0$ are constructed by a ``shift-and-scale" procedure so that each $\vartheta^0_a$ is dominated by the exact term $\frac{1}{\varepsilon} dH$. This strategy to recover the ordinary Novikov Principle has been known among experts for quite awhile, see \cite[Page 302]{pajitnov1995} for a historical account, \cite[Page 548]{Ono2005} and \cite[Theorem 3.5.2]{li2009}. However, our approach is slightly different as it does not make use of gradient like vector fields.

We conclude the present subsection by explaining how to recover \cite[Theorem 5.1]{Pajitnov2019} from the polytope Novikov principle. For the reader's convenience we briefly recall Pajitnov's setting, keeping the notation as close as possible to \cite{Pajitnov2019}. Fix a Morse-Smale pair $(\omega,g)$ on $M$ and let $p \colon \widehat{M} \to M$ be a regular cover such that $p^*[\omega]=0$. Denote by $r$ the rank of $\omega$\footnote{The rank of a cohomology class $a \in H^1_{\mathrm{dR}}(M)$ is defined as $\mathrm{rank}_{\Z}\left(\mathrm{im} \, \Phi_a\right)$.} and define $G=\mathrm{Deck}(\widehat{M})$. Viewing the period homomorphism $\Phi_\omega$ on $H_1(M;\Z)$ we get a splitting $$H_1(M;\Z) \cong \Z^r \oplus \ker[\omega].$$ Pajitnov calls a family of homomorhism $$\Psi_1,\dots, \Psi_r \colon \Z^r \to\Z$$ a $\Phi_\omega$-regular family if
\begin{itemize}
\item the $\Psi_i$ span $\hom_\Z(\Z^r,\Z)$ and
\item the coordinates of $\Phi_\omega \colon \Z^r \to \R$ in the basis $\Psi_i$ are strictly positive.
\end{itemize}
We shall call $\Psi=\lbrace \Psi_i \rbrace$ a $\Phi_\omega$-\emph{semi-regular} family whenever the first bullet point above is satisfied. One should not be fooled by the length of name -- the existence of a semi-regular family is obvious and merely an algebraic statement.

To every (semi)-regular family $\Psi=\{\Psi_1,\dots,\Psi_r\}$ we associate the conical Novikov ring
$$\widehat{\Lambda}_\Psi=\bigcap_{i=1}^r \widehat{\Z}[G]^{\Psi_i}.$$ The conical Novikov chain complex $(\mathcal{N}_\bullet(\omega),\partial)$ is defined as
$$\mathcal{N}_i(\omega,\Psi)=\mathcal{N}_i(\omega)=\bigoplus_{x \in Z_i(\omega)} \widehat{\Lambda}_{\Psi} \langle x \rangle,$$ where the boundary operator $\partial$ is defined as expected: fix preferred lifts $\hat{x}$ of each $x \in Z(\omega)$ and define the $y$-component of $\partial x$ by the (signed) count of $\hat{f}_\omega$-Morse flow lines on the cover $\widehat{M}$ from $\hat{x}$ to $g \circ \hat{y}$ for all $g \in G$.\footnote{The definition in \cite{Pajitnov2019} is slightly more general, as they define the count with respect to any transverse $\omega$-gradient.} Pajitnov proves that there always exists a $\Phi_\omega$-regular family $\Psi$ so that $(\mathcal{N}_\bullet(\omega),\partial)$ is a well defined $\widehat{\Lambda}_\Psi$-module chain complex and shows:

\begin{theorem}[Pajitnov 2019, \cite{Pajitnov2019}]\label{theorem Pajitnov}
For any Morse-Smale pair $(\omega,g)$ there exists a $\Phi_\omega$-regular family $\Psi$ such that $\mathcal{N}_\bullet(\omega,\Psi)=\mathcal{N}_\bullet(\omega)$ is a well defined chain complex. Moreover, for any such $\Psi$ it holds
$$\mathcal{N}_\bullet(\omega)  \simeq C_\bullet\left(\widehat{M}\right) \otimes_{\Z[G]} \widehat{\Lambda}_{\Psi}$$ as $\widehat{\Lambda}_{\Psi}$-modules.
\end{theorem}

Using Theorem \ref{theorem Polytop Novikov Principle} we recover Theorem \ref{theorem Pajitnov} in the abelian case.

\begin{corollary}\label{corollary pajitnov NP}
Let $(\omega,g)$ be a Morse-Smale pair, $p \colon \widehat{M} \to M$ be an abelian regular cover with $p^*[\omega]=0$. 

Then there exists $\Phi_\omega$-semi-regular family $\Psi$, a section $\theta \colon \mathcal{A} \to \Omega^1(M)$ around $[\omega]$ with $\widetilde{M}_\mathcal{A}=\widehat{M}$, and a perturbation $\vartheta \colon \mathcal{A} \to \Omega^1(M)$ such that
$$\mathrm{CN}_\bullet(\vartheta_a,\mathcal{A}|_{\mathcal{B}}) \simeq C_\bullet \left( \widehat{M} \right) \otimes_{\Z[G]} \widehat{\Lambda}_{\Psi}, \quad \forall a \in \mathcal{A}$$ as Novikov-modules with $\vartheta_{[\omega]}=\omega$. In particular,
$$\mathrm{CN}_\bullet(\omega,\mathcal{A}|_{\mathcal{B}}) \simeq C_\bullet \left( \widehat{M} \right) \otimes_{\Z[G]} \widehat{\Lambda}_{\Psi}.$$
\end{corollary}

\begin{proof}

First of all we construct the polytope $\mathcal{A}$ and a small section $\theta\colon \mathcal{A} \to \Omega^1(M)$ by adapting a rational approximation idea due to Pajitnov \cite{Pajitnov1996}, see also \cite[Section 4.2]{Schütz2002}.
Consider the splitting from before 
$$H_1(M;\Z) \cong \Z^r \oplus \ker[\omega],$$ where $r$ is the rank of $[\omega]$. Pick $r$-many generators $\gamma_1,\dots,\gamma_r$ of the first summand above. By de Rham's Theorem there $r$-many pairwise distinct integral classes $a_1,\dots,a_r$ dual to $\gamma_1,\dots,\gamma_r$ such that 
$$\ker(a_l) \supset \ker[\omega], \quad \forall l=1,\dots,r.$$ In particular, we can write $$\Phi_\omega= \sum_{l=1}^r u_l \cdot \Phi_{a_l},$$ for a unique vector $ \underline{u}=(u_1,\dots,u_r) \in \R^r$, hence $[\omega]=\sum_{l=1}^r u_l \cdot a_l$. Thus, for our fixed Morse representative $\omega \in [\omega]$ there exist $\alpha_l \in a_l$ with
\begin{equation}\label{equation sum and kernel}
\omega=\sum_{l=1}^r u_l \cdot \alpha_l \text{ and }  \ker[\omega]=\bigcap_{l=1}^r \ker(a_l).
\end{equation} 
For fixed $\varepsilon>0$ we construct $r$-many \emph{rational} closed one-forms $\beta_l$ that are $\varepsilon$-close to $\omega$ (in the operator norm induced by $g$). Pick a sufficiently small vector $\underline{v}^1=\underline{v}^1(\varepsilon) \in \R^r$ such that for
$$\beta_1=\omega+\sum_{l=1}^r v^1_l \cdot \alpha_l$$ we have
\begin{equation*}
\Vert \omega - \beta_1 \Vert \leq \sum_{l=1}^r \vert v^1_l \vert  \cdot \Vert \alpha_l \Vert < \varepsilon \text{ and } b_1:=[\beta_1] \in H^1(M;\Q).
\end{equation*}
This is possible since $\Q$ is dense in $\R$. Define 
$$\beta_2=\omega+\sum_{l=1}^r v_l^2 \cdot \alpha_l, \, \underline{v}^2 \in \R^r,$$
satisfying the same properties with $$\underline{v}^2-\underline{v}^1=(\underbrace{v^2_1-v^1_1}_{\neq 0},0,\dots,0).$$ In particular, $$\beta_2-\beta_1=(v_1^2-v_1^1) \cdot \alpha_1,$$ which implies $b_1 \neq b_2$. Proceeding inductively (e.g. $\underline{v}^3$ agreeing with $\underline{v}^2$ except for the second entry etc.) we end up with $r$-many rational one-forms $\beta_1,\dots,\beta_r$ satisfying
\begin{itemize}
\item $\Vert \omega - \beta_j \Vert \leq \sum_{l=1}^{r} \vert v^j_l \vert \cdot \Vert \alpha_l \Vert < \varepsilon$ for all $j=1,\dots,k$,
\item  $b_l \neq b_j$ for all $l \neq j$. 
\item $\ker(b_l) \supset \ker[\omega]$ for all $l=1,\dots,k$, by \eqref{equation sum and kernel}.
\end{itemize}
Since all $b_l$ are rational, there exists a positive integer $q \in \N$ so that every cohomology class $q \cdot b_l$ is integral. From the bullet points above we thus conclude that $$\Psi:=\{ \Phi_{q \cdot b_l}\}$$ defines a $\Phi_\omega$-semi-regular family. Next set $\mathcal{A}^-=\langle [\omega],b_1,\dots,b_r \rangle$ and  $$\varepsilon=\frac{C_\omega}{D \cdot 1000},$$ cf. Corollary \ref{corollary other chain complex perturbed}. It is clear that there exists a section $\theta^- \colon \mathcal{A}^- \to \Omega^1(M)$ sending $[\omega]$ to $\omega$ and $b_l$ to $\beta_l$. In particular
$$\Vert \theta^-_a- \omega \Vert < \varepsilon, \quad \forall a \in \mathcal{A}^-.$$ Analogously to the construction of the $\beta_l$'s, one can define rational one-forms close to $\omega$ that do not vanish on $\ker[\omega]$. Including some of those cohomology classes allows us to extend $\mathcal{A}^-$ to $\mathcal{A}$ so that $\widetilde{M}_{\mathcal{A}}=\widehat{M}$ with a section $\theta \colon \mathcal{K} \to \Omega^1(M)$ that extends $\theta^-$ and is still $\varepsilon$-close to $\omega$. By choice of $\varepsilon$ we can invoke Corollary \ref{corollary other chain complex perturbed} and define a perturbed section $$\vartheta^\omega \colon \mathcal{A} \to \Omega^1(M)$$ with $(\omega,g)$ as the underlying reference pair.
But now we are in a position to use Theorem \ref{theorem Polytop Novikov Principle} and Remark \ref{remark PNP for small sections is GOOD} with $\mathcal{B}:=\langle b_1,\dots,b_r \rangle$ to obtain
\begin{align*}
\mathrm{CN}_\bullet(\vartheta^\omega_a,\mathcal{A}|_{\mathcal{B}}) \simeq C_\bullet(\widehat{M}) \otimes_{\Z[G]} \mathrm{Nov}(\mathcal{A}|_{\mathcal{B}}) = C_\bullet(\widehat{M}) \otimes_{\Z[G]} \bigcap_{l=1}^r\mathrm{Nov}(\mathcal{A}|_{b_l}).
\end{align*}
The last equality follows from the $\Z$-analogue of Lemma \ref{lemma vertex determien CN polytope}. Note that $\mathrm{Nov}(\mathcal{K}|_{b_l})=\widehat{\Z}[G]^{b_l} = \widehat{\Z}[G]^{q \cdot b_l}$, therefore the Novikov ring on the RHS above does coincide with $\widehat{\Lambda}_{\Psi}$. By definition of the perturbation $\vartheta^\omega$ we get $\vartheta^\omega_{[\omega]}=\omega$, which finally concludes the proof.

\end{proof}

\bibliographystyle{abbrv}
\bibliography{Bibliography_Master}

\end{document}